\documentclass[12pt,a4paper]{amsart}

\usepackage[latin1]{inputenc}
\usepackage{amssymb, amsmath, amscd, psfrag, amsthm}
\usepackage{mathrsfs}
\usepackage[all]{xy}
\usepackage[T1]{fontenc}
\usepackage{graphicx}
\usepackage[noadjust]{cite}

\newtheorem{theorem}{Theorem}[section]
\newtheorem{proposition}[theorem]{Proposition}
\newtheorem{corollary}[theorem]{Corollary}
\newtheorem{lemma}[theorem]{Lemma}
\newtheorem*{ThM}{Main Theorem}

\theoremstyle{definition}
\newtheorem{definition}[theorem]{Definition}

\theoremstyle{remark}
\newtheorem{remark}[theorem]{Remark}

\newcommand{\CC}{{\mathcal C}}
\newcommand{\cD}{{\mathcal D}} 
\newcommand{\CK}{{\mathcal K}}
\newcommand{\CM}{{\mathcal M}}
\newcommand{\CR}{{\mathcal R}}
\newcommand{\CS}{{\mathcal S}}
\newcommand{\CT}{{\mathcal T}}
\newcommand{\CW}{{\mathcal W}}
\newcommand{\SA}{{\mathscr A}}
\newcommand{\SF}{{\mathscr F}}
\newcommand{\ZZ}{{\mathbb Z}}
\newcommand{\NN}{{\mathbb N}}
\newcommand{\QQ}{{\mathbb Q}}
\newcommand{\ch}{{\operatorname{char}\, }}
\newcommand{\Hom}{{\operatorname{Hom}}}
\newcommand{\Lie}{{\operatorname{Lie}}}
\newcommand{\id}{{\operatorname{id}}}
\newcommand{\supp}{{\operatorname{supp}}}
\newcommand{\catmod}{{\operatorname{-mod}}}
\newcommand{\rk}{{{\operatorname{rk}}}}
\newcommand{\GL}{{\operatorname{GL}}}
\newcommand{\Aff}{{\operatorname{Aff}}}
\newcommand{\pr}{{\operatorname{pr}}}
\newcommand{\sign}{{\operatorname{sign}}}
\newcommand{\ul}{\underline}
\newcommand{\w}[2]{#1^{(#2)}}
\renewcommand{\wp}[2]{#1^{(#2)}_{\, +}}
\newcommand{\wm}[2]{#1^{(#2)}_{\, -}}

\numberwithin{equation}{section}

\begin{document}

\title{Duality in the Category of Andersen-Jantzen-Soergel}
\author{Friederike Steglich*}
\address{Department Mathematik, Friedrich-Alexander-Universit\"at Erlangen-N\"urnberg, Cauerstr. 11, 91058 Erlangen, Germany}
\email{steglich@mi.uni-erlangen.de}
\thanks{*supported by the DFG priority program 1388}

\begin{abstract}
In the early 1990's Andersen, Jantzen and Soergel introduced a  category in order to give a combinatorial model for certain  representations of quantum groups at a root of unity and 	simultaneously of Lie algebras of semisimple algebraic groups in positive characteristic. We will describe the behaviour of duality in this category. 
\end{abstract}

\maketitle

\section{Introduction}
In \cite{MR1272539}, Andersen, Jantzen and Soergel introduced a combinatorial model to describe certain representations of the two following situations:
\begin{enumerate}
\item $U_p$ is the quantized enveloping algebra of a complex Lie algebra at a $p$-th root of unity, where $p>1$ is an odd integer (and prime to 3 if the corresponding root system is of type $G_2$),
\item $\Lie\, G_k$, where $G_k$ is a connected, simply connected semisimple algebraic group over an algebraically closed field $k$ of char $p>h$, where $h$ is the Coxeter number of the corresponding root system. 
\end{enumerate}
A relevant consequence of this work is the following. Lusztig's Modular Conjecture (cf. \cite{lusztig1980some}) provides a character formula for representations of connected reductive algebraic groups over an algebraically closed field $k$ with $\ch k =p>h$. In this context, \cite{MR1272539} designates an essential step in proving this conjecture for almost all $p$ by deducing this from the validity of the char 0 analogue, Lusztig's Quantum Conjecture (resulting from \cite{zbMATH00755403} and \cite{KL93}).

This model is realized by a combinatorial category $\CK_k$ and its full subcategory of so-called \emph{special objects} $\CM_k'$.  $\CM_k'$ is then equivalent to $\CR_k$, the category of \emph{$X$-graded special deformed modular representations} of $\Lie \, G_k$ where $X$ is the corresponding weight lattice. The feature of $\CR_k$ is to contain the objects that are of primary interest regarding a character formula, namely the deformed projective objects of the principal block of the category of \emph{$X$-graded restricted representations}. 

The only necessary data derived from $G_k$ to construct a certain $X$-graded $\CM_k'$ is the corresponding root datum. Starting with a unit object, we obtain this category $\CM_k'$ by repeatedly  applying a set of translation functors $\CT_s'$ indexed by the simple affine reflections, taking direct sums and direct summands. However, the main actor of this paper is the category $\CM_k$ that is constructed in the same manner with the only difference lying in the definition of the translation functors $\CT_s$. This definition is related to moment graph theory and due to Fiebig who showed in \cite{MR2726602} that $\CM_k$ and $\CM_k'$ are equivalent categories. The methods in \cite{MR2726602} lead to an improvement in the matter of Lusztig's Modular Conjecture by giving an explicit bound on exceptional primes. Furthermore, a proof of Lusztig's Quantum Conjecture independent from \cite{zbMATH00755403}, \cite{KL93} could be given. 

It is very difficult to describe these special objects intrinsically. A first step towards such a description is -- and that is what is done in this paper -- to study the behaviour of an intuitionally defined duality in the category $\CM_k$ and its parent $\CK_k$. The main result of this paper is
\begin{ThM}[Theorem \ref{selfdualanti}]
The indecomposable special objects in $\CM_k$ that correspond to alcoves in the anti-fundamental box are tilting self-dual up to a shift.
\end{ThM}

\section{Preliminaries}

\subsection{Notation}\label{sec:not}
Let $R$ be a reduced and irreducible root system in a finite dimensional $\QQ$-vector space $V$.  Furthermore, the coroot system of $R$ is denoted by $R^{\vee} \subset V^*=\Hom_{\QQ}(V,\QQ)$, $R^{+}\subset R$ is a subset of chosen positive roots and this determines a set of simple roots $\Delta \subset R^{+} \subset R$. 

For each $\alpha \in R$ and $n \in \ZZ$, an affine transformation in $\Aff (V^*)$ is defined by
\begin{equation*}
	s_{\alpha, n}(v)=v-(\langle\alpha, v\rangle -n)\alpha^{\vee},
\end{equation*}
where $\langle\cdot,\cdot\rangle\colon V \times V^* \to \QQ$ is the canonical pairing. This transformation is a reflection through the hyperplane 
\begin{equation*}
	H_{\alpha, n}=\{v\in V^*\mid \langle\alpha, v\rangle=n\},
\end{equation*}
and therefore, it partitions the dual space into three parts
\begin{equation}\label{hyperpartition}
	V^*=H_{\alpha,n}^-\dot\cup H_{\alpha,n}\dot\cup H_{\alpha,n}^+
\end{equation}
with
\begin{equation*}
H_{\alpha,n}^-=\{v\in V^*\mid \langle\alpha, v\rangle<n\}
\end{equation*}
and
\begin{equation*}
H_{\alpha,n}^+=\{v\in V^*\mid \langle\alpha, v\rangle>n\}. 
\end{equation*}

Let $\tilde{\alpha}$ be the highest root in $R$, i.e. the unique root with the property that for every $\alpha \in R^{+}$, 
\begin{equation*}
	\tilde{\alpha} -\alpha \in\sum_{\beta \in \Delta}\ZZ_{\geq 0} \beta.
\end{equation*}
Then 
\begin{equation*}
	\widehat{\CS}:=\{s_{\alpha,0}\mid \alpha \in \Delta\}\cup \{s_{\tilde{\alpha},1}\}
\end{equation*}
is called the set of \emph{simple affine reflections}. We will distinguish between the \emph{finite Weyl group} $\CW \subset \GL (V^*)$ generated by $\CS:=\{s_{\alpha,0}\mid \alpha \in \Delta\}$ and the \emph{affine Weyl group} $\widehat{\CW}\subset\Aff(V^*)$ generated by $\widehat{\CS}$. As  $(\widehat{\CW},\widehat{\CS})$ is a Coxeter system, one is able to  associate a length function 
\begin{equation*}
l\colon \widehat{\CW} \to \NN.
\end{equation*}
The element of maximal length of the finite Coxeter group $\CW$ is denoted by $w_0$. 

There is a common way to identify elements in $\widehat{\CW}$ with subsets of $V^{*}$. One considers the coarsest partition $\SF$ of $V^{*}$ that refines all partitions of the form \eqref{hyperpartition} for $\alpha \in R^+, n \in \ZZ$. The open components of this partition are known as \emph{alcoves} with the designated \emph{fundamental alcove} 
\begin{equation*}
A_e=\{v\in V^*\mid 0<\langle \alpha, v\rangle <1 \quad\forall \alpha \in R^+\}. 
\end{equation*}
$\widehat{\CW}$ acts simply transitively on the \emph{set of alcoves} $\SA$, and hence there is the bijection
\begin{align*}
\widehat{\CW} &\to \SA\\
w&\mapsto A_w=w.A_e.
\end{align*}
A component $B$ of $\SF$ which is included in exactly one hyperplane $H_{\alpha,n}$ is called \emph{wall}. In this situation, the \emph{type $\alpha(B)$ of} $B$ is the positive root $\alpha$. A wall $B$ is a \emph{wall to} an alcove $A\in \SA$ if $B$ is a subset of the closure of $A$. $B_-$ (resp. $B_+$) denotes the unique alcove whose wall is $B$ and which is contained in the negative (resp. positive) half-space corresponding to $B$. 

For every $\beta \in R^{+}$ there is a bijection $\beta \uparrow\cdot$ on $\SF$ mapping alcoves to alcoves and walls to walls. For $F\in \SF$ the component $\beta\uparrow F$ is given by $s_{\beta,m}.F$ where $m$ is the smallest integer such that $F\subset H^-_{\beta,m}\cup H_{\beta,m}$. The inverse of $\beta\uparrow\cdot$ is denoted by $\beta \downarrow\cdot$.

To each simple affine reflection $s\in \widehat{\CS}$ we associate the wall $A_{s,e}$ of the fundamental alcove $A_e$ which is contained in the closure of $A_s$, too. Thereby, we obtain a mapping 
\begin{align*}
\SA \to \{\text{walls}\}\\
A\mapsto \w{A}{s}
\end{align*}
where $\w{A}{s}$ is the unique wall which lies in the $\widehat{\CW}$-orbit of $A_{s,e}$ and intersects the closure of $A$. In case we need the negative alcove to a wall of an alcove, namely $(\w{A}{s})_-$ we will simply denote it by $\wm{A}{s}$ and analougously $(\w{A}{s})_+=\wp{A}{s}$. This should not cause any confusion. For having a more detailed view on alcoves, \cite{humphreys1990reflection} is recommended. 

Take an algebraically closed field $k$ with $\ch k\neq 2$ and, if $R$ is of type $G_2$, $\ch k \neq 3$. Let 
\begin{equation*}
X=\{v\in V\mid \langle v, \alpha^{\vee}\rangle \in \ZZ \quad \forall \alpha \in R \}
\end{equation*}
and let $S_k$ be the symmetric algebra of $X_k:=X \otimes_{\ZZ}k$. We define the following $S_k$-subalgebras of the quotient field of $S_k$,
\begin{equation*}
S_k^{\emptyset}=S_k[\alpha^{-1}\mid\alpha\in R^+]
\end{equation*}
and 
\begin{equation*}
S_k^{\beta}=S_k[\alpha^{-1}\mid\alpha \in R^+, \alpha\neq \beta]
\end{equation*} 
for $\beta \in R^+$.
All those algebras carry a compatible $\ZZ$-grading which is induced by setting the degree of $X$ to $2$.

\subsection{Alcoves}

In the course of this article we will need to be familiar with the behaviour of alcoves and their walls in connection with the action of the affine Weyl group. The necessary properties are collected in this paragraph.

\begin{lemma}[\cite{MR2726602}, Lemma 5.1]\label{wallcomb} Let $s\in\widehat{\CS}$ and $\beta\in R^+$. Choose $A\in\SA$ and $B\in \widehat{\CW}.A_{s,e}$.
\begin{enumerate}
\item[a)] If $\beta\uparrow B=B$, then $\beta\uparrow B_-=B_+$. 
\item[b)] If $\beta\uparrow B\ne B$, then $\{\beta\uparrow B_-, \beta\uparrow B_+\}=\{(\beta\uparrow B)_-, (\beta\uparrow B)_+\}$.
\item[c)] If $\beta\uparrow \w{A}{s}=\w{A}{s}$ and $A=\wm{A}{s}$, then $\w{(\beta\uparrow A)}{s}=\w{A}{s}$ and $\beta\uparrow A=\wp{A}{s}$.
\item[d)] If $\beta\uparrow \w{A}{s}\ne \w{A}{s}$, then $\w{(\beta\uparrow A)}{s}=\beta\uparrow \w{A}{s}$.
\end{enumerate}
\end{lemma}

\begin{remark}\label{updown}
If $\w{A}{s}$ is of type $\beta$ with $\beta \uparrow A=s_{\beta,n}.A$, then $\w{(\beta \uparrow A)}{s}=(s_{\beta, n}.A)^{(s)}=s_{\beta, n}.\w{A}{s}$ is of type $\beta$, too, and vice versa. In formula this means
\begin{equation*}
\beta \uparrow \w{A}{s}=\w{A}{s} \iff \beta\uparrow\w{(\beta \uparrow A)}{s}=\w{(\beta \uparrow A)}{s}.
\end{equation*}

Let $\beta \uparrow \w{A}{s}=\w{A}{s}$ and $A=\wp{A}{s}$. In particular, it follows from  Lemma \ref{wallcomb}.a) that 
\begin{equation*}
	\beta \downarrow A= \wm{A}{s}
\end{equation*}
and 
\begin{equation*}
\beta \uparrow A =\wm{(\beta \uparrow A)}{s}.
\end{equation*}
\end{remark}

Below, there are two more lemmas on alcove calculation with an involved finite Weyl group element. Define $w(\beta)^+$ to be the unique element in $\{w(\beta), -w(\beta)\}\cap R^+$.
\begin{lemma}\label{kipp:wb}
Let $\beta \in R^{+}, A \in \SA$ and $w \in \CW$. Then
\begin{equation*}
	w.(\beta\uparrow A) =
	\begin{cases}
		w(\beta)^+\uparrow w.A & \text{if }w(\beta) \in R^{+},\\
		w(\beta)^+\downarrow w.A & \text{if }w(\beta) \in R^{-}.
	\end{cases}
\end{equation*}
\end{lemma}

\begin{proof}
Let $n \in \ZZ$ be such that 
\begin{equation}\label{prel:betaup}
	\beta \uparrow A =s_{\beta,n}.A
\end{equation}
and hence 
\begin{equation*}
w.(\beta \uparrow A) = (w s_{\beta, n}).A= s_{w(\beta),n}.(w.A) = s_{-w(\beta),-n}.(w.A).
\end{equation*}
Equation \eqref{prel:betaup} is equivalent to 
	\begin{equation*}
\forall N\geq n:\ A \subset H_{\beta,N}^{-} \quad\text{ and }
\quad \forall M < n: \, A\subset H_{\beta,M}^{-}.
	\end{equation*}
But
\begin{equation*}
A \subset H_{\beta,N}^{-} \iff w.A \subset w(H_{\beta,N}^{-})= H_{w(\beta),N}^{-}=H_{-w(\beta),-N}^{+}.
\end{equation*}
If $w(\beta)\in R^+$, this leads to $w.(\beta \uparrow A)= s_{w(\beta),n}.(w.A)=w(\beta)^+ \uparrow w.A$. 
Furthermore, 
\begin{align*}
A \subset H_{\beta,N}^{-}& \\
\iff\quad &w.(\beta \uparrow A) \subset w s_{\beta, n} (H_{\beta,N}^{-})\\
&= w(H_{\beta,2n-N}^{+})=H_{-w(\beta),N-2n}^{-}.
\end{align*}
Hence for $w(\beta)\in R^-$, we obtain
\begin{equation*}
w(\beta)^+ \uparrow w.(\beta \uparrow A)=s_{-w(\beta), -n}w.(\beta \uparrow A)= w.A.
\end{equation*}
This yields the claim. 
\end{proof}

\begin{lemma}\label{kipp:arithmetik}
	Let $w\in \CW$ then
	\begin{enumerate}
	\item[a)] $\{\wm{(w.A)}{s}, \wp{(w.A)}{s}\}=\{w.(\wm{A}{s}),w.(\wp{A}{s})\}$,
	\item[b)] $\beta \uparrow \w{A}{s} =\w{A}{s}$ if and only if $w(\beta)^+\uparrow
	\w{(w.A)}{s}=\w{(w.A)}{s}$. 
	\end{enumerate}
\end{lemma}
\begin{proof}
\begin{enumerate}
	\item[a)] The common wall of the left-hand set is $\w{(w.A)}{s}$ and $w.\w{A}{s}$ is the common wall of the right-hand set. But $w.\w{A}{s}=\w{(w.A)}{s}$ and hence, both sets coincide. (This is even true for $w\in \widehat{\CW}$.)
	\item[b)] Let $\beta \uparrow \w{A}{s} =\w{A}{s}$ with  $\w{A}{s} \subset H_{\beta, n}$. Applying $w$ yields  $w.\w{A}{s}=\w{(w.A)}{s}\subset H_{w(\beta), n}$ and therefore $w(\beta)^+\uparrow \w{(w.A)}{s}=\w{(w.A)}{s}$. 
	For the if-part we simply set $w^{-1}$ into the role of $w$. 
	\end{enumerate}
\end{proof}

\section{The Category of Andersen-Jantzen-Soergel}

\subsection{The Category $\CK_k$}

As a first step, we will define the category $\CK_k$. Objects in $\CK_k$ are of the form $M=(\{M(A)\}_{A\in \SA}, \{M(A,\beta)\}_{A\in \SA, \beta \in R^{+}})$, where
\begin{enumerate}
	\item for all $A\in \SA$: $M(A)$ is a $\ZZ$-graded $S^{\emptyset}_k$-module, 
	\item for all $A\in \SA, \beta \in R^{+}$: $M(A,\beta)$ is an $S^{\beta}_k$-submodule of $M(A)\oplus M(\beta \uparrow A)$ with compatible grading.
\end{enumerate}
A morphism $f\colon M \to N$ in $\CK_k$ is given by $f= (f_A)_{A\in \SA}$, where
\begin{enumerate}
	\item each $f_A\colon M(A) \to N(A)$ is a graded $S^{\emptyset}_k$-module homomorphism of degree $0$, 
	\item for all $A \in \SA$ and $\beta \in R^{+}$ the homomorphism $f_A\oplus f_{\beta\uparrow A}$ maps the $S^{\beta}_k$-submodule $M(A,\beta)$ of \mbox{$M(A)\oplus M(\beta \uparrow A)$} to $N(A, \beta)\subseteq N(A) \oplus N(\beta \uparrow A)$.
\end{enumerate}

\subsubsection*{Examples}
We want to introduce two examples of objects in $\CK_k$. The first, called the \emph{unit object}, is defined by
\begin{align*}
P_0(A)& :=  
\begin{cases} S^{\emptyset}_k & \text{if }A=A_e, \\
0 & \text{else},
\end{cases}\\
P_0(A,\beta)& := \begin{cases} S^{\beta}_k & \text{if }A\in \{A_e,\beta \downarrow A_e\}, \\
0 & \text{else}.
\end{cases}\\
\end{align*}

To describe the second example, we need to define a set $\SA^{-}_{\beta}\subset \SA$ for any positive root $\beta$. This is the set of those alcoves inside the negative half-space of $\beta$ that additionally correspond to elements of the finite Weyl group. In formula this means
\begin{equation}\label{abminus}
	\SA^{-}_{\beta}=\{A \in \SA\mid A=A_w \text{ with } w\in \CW, A \subset H_{\beta}^{-}\}.
\end{equation}
For a graphical description, see Figure \ref{fig:cases2}. 

\begin{figure}[htbp] 
  \centering
  \includegraphics[width=0.4\textwidth]{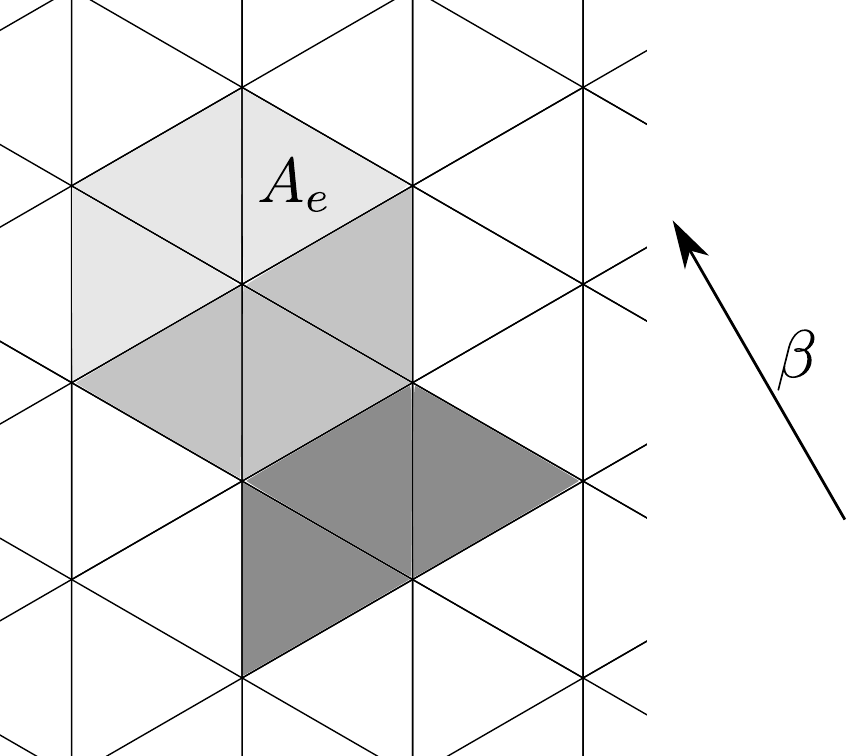}
  \caption{The following subsets of $\SA$, here of type $\tilde{A}_2$:}
  Color key \quad
\includegraphics{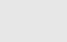} \quad $A\in \SA_+^{\beta}$ \qquad
\includegraphics{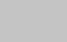} \quad  $A\in \SA_-^{\beta}$ \qquad
\includegraphics{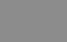} \quad  $A\in \beta\downarrow \SA_-^{\beta}$  
  \label{fig:cases2}
\end{figure}

Analogously, we define the set $\SA^{+}_{\beta}$. Please notice that 
\begin{equation*}
	\SA^{+}_{\beta}=\beta \uparrow \SA^{-}_{\beta}.
\end{equation*}
We define the object $Q_0 \in \CK_k$ as follows. 
\begin{align*}
Q_0(A) & :=  
\begin{cases} S^{\emptyset}_k & \text{if $A=A_w$ with } w \in \CW, \\
0 & \text{else},
\end{cases}\\
Q_0(A,\beta) & := 
\begin{cases} \{(\beta x+y, y)\mid x,y \in S^{\beta}_k\} & \text{ for } A \in \SA^{-}_{\beta},\\ 
\beta S^{\beta}_k & \text{ for } A \in \beta\uparrow \SA^{-}_{\beta},\\
S^{\beta}_k & \text{ for } A \in \beta \downarrow \SA^{-}_{\beta}, \\
0 & \text{ else}.
\end{cases}
\end{align*}
$Q_0$ is indeed an object in $\CK_k$ as there are the following embeddings:
\begin{eqnarray*}
	\forall A \in \SA^{-}_{\beta} \colon &\{(\beta x+y, y)\mid x,y \in S^{\beta}_k\}&\hookrightarrow S^{\emptyset}_k \oplus S^{\emptyset}_k,\\
	\forall A \in \beta \uparrow\SA^{-}_{\beta} \colon &\beta S^{\beta}_k\oplus 0&\hookrightarrow S^{\emptyset}_k \oplus 0,\\
	\forall A \in \beta \downarrow \SA^{-}_{\beta} \colon &0\oplus S^{\beta}_k &\hookrightarrow 0\oplus S^{\emptyset}_k.
\end{eqnarray*}

\subsection{The Categories $\CM_k$ and $\CM_k^{\circ}$}

\subsubsection{Translation Functors}
In the work \cite{MR1272539} by Andersen, Jantzen and Soergel, a set of translation functors indexed by the affine simple roots was introduced in order to span a subcategory of $\CK_k$. In \cite{MR2726602} it was then shown that one is able to replace those original translation functors by different translation functors that yield an equivalent subcategory. Fiebig's translation functors are motivated by moment graph theory. We will work with the translation functors by Fiebig. 

The \emph{translation functor $\CT_s\colon \CK_k \to \CK_k$ associated with} a simple affine reflection $s\in \widehat{\CS}$ is given by
\begin{align*}
\CT_s M (A) &= M(\wm{A}{s}) \oplus M(\wp{A}{s}),\\
\CT_s M (A,\beta) & =  
\begin{cases} 
	\{(\beta x + y,y)\mid x,y \in M(A,\beta)\} & \text{if $\beta \uparrow \w{A}{s}=\w{A}{s}$}\\
	&\quad \text{and $A=\wm{A}{s}$},\\
	\beta M(\beta \downarrow A,\beta) \oplus M(\beta \uparrow A,\beta) & \text{if $\beta \uparrow \w{A}{s}=\w{A}{s}$}\\
	&\quad \text{and $A=\wp{A}{s}$},\\
	M(\wm{A}{s},\beta) \oplus M(\wp{A}{s},\beta)&\text{if }\beta \uparrow \w{A}{s}\neq\w{A}{s},
\end{cases}
\end{align*}
where $A\in \SA$ and $\beta \in R^+$. See Figure \ref{fig:cases} for a graphical description of the three different cases appearing in the definition of $\CT_s M(A,\beta)$. The translation functors are well-defined due to Lemma \ref{wallcomb} and Remark \ref{updown}. For complete description of $\CT_s$, we must specify how the $S^{\beta}_k$-submodules $\CT_s M (A,\beta)$ embed into $\CT_s M (A)\oplus \CT_s M (\beta \uparrow A)$. If $\beta \uparrow \w{A}{s}=\w{A}{s}, A= \wm{A}{s}$, the first entry of a pair in $\CT_s M (A,\beta)$ maps to $\CT_s M (A)$ and the second to $\CT_s M(\beta \uparrow A)$. If $\beta \uparrow \w{A}{s}=\w{A}{s}, A= \wp{A}{s}$, the first direct summand of $\CT_s M (A, \beta)$ lives inside $\CT_s M (A)$ and the second inside $\CT_s M (\beta \uparrow A)$. For $\beta \uparrow \w{A}{s}\neq \w{A}{s}$, the submodule $M(\wm{A}{s},\beta)$ lies crosswise inside both $\CT_s M (A)$ and $\CT_s M (\beta \uparrow A)$ and analogously $M(\wp{A}{s},\beta)$ does. 

\begin{figure}[htbp] 
  \centering
     \includegraphics[width=\textwidth]{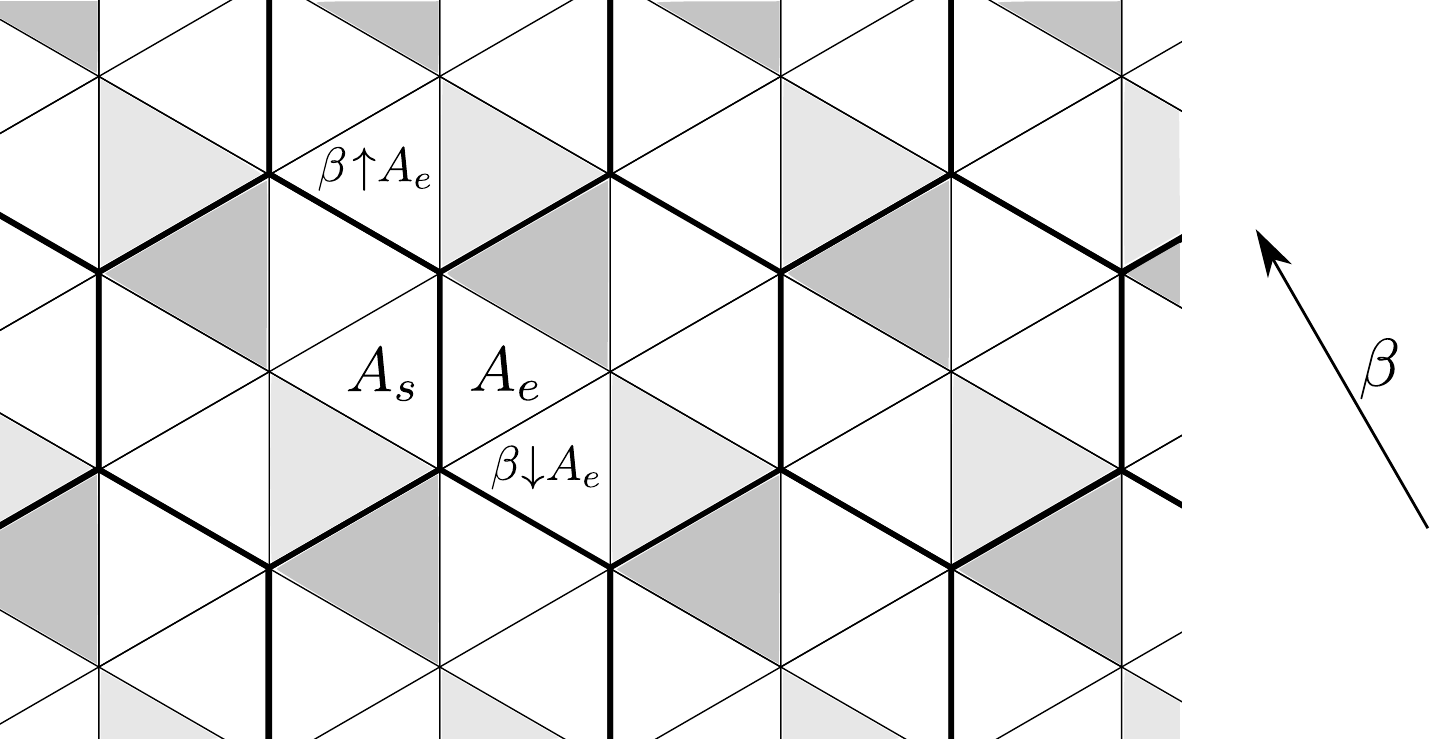}
  \caption{Example of an alcove pattern, here of type $\tilde{A}_2$}
  \label{fig:cases}
\flushleft The color key of the three different cases:
  \begin{enumerate}
  \item \includegraphics{artcp1.pdf}\qquad $\beta \uparrow \w{A}{s}=\w{A}{s}, A= \wm{A}{s}$,
  \item \includegraphics{artcp2.pdf}\qquad $\beta \uparrow \w{A}{s}=\w{A}{s}, A= \wp{A}{s}$,
  \item \includegraphics{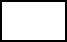}\qquad $\beta \uparrow \w{A}{s}\neq \w{A}{s}$ .
  \end{enumerate}
\end{figure}

For an integer $l$ we define the \emph{shift functor} 
\begin{equation*}
\{l\}\colon \CK_k \to \CK_k
\end{equation*}
to be the functor that shifts an object $M$ in all components $M(A)$ and $M(A,\beta)$ by $l$. Here, the homogeneous elements of degree $n$ of an $\{l\}$-shifted graded $S^{\emptyset}_k$- resp. $S^{\beta}_k$-module are given by 
\begin{equation*}
	M\{l\}_{n}=M_{n-l}.
\end{equation*}
The shift functor is certainly not superfluous. Indeed each $A$-com\-ponent separately is isomorphic to the $\{-2\}$-shifted component itself, where the isomorphism is given by a multiplication with an $\alpha \in R^+$ that is a unit in $S^{\emptyset}_k$. But this does not yield a global isomorphism on the whole object in general as $\alpha$ is not invertible in $S^{\alpha}_k$.

\subsubsection{Special objects}
An object $M \in \CK_k$ is called \emph{special} if it is isomorphic to a direct summand of a direct sum of objects of the form
\begin{equation}\label{eq:bottsamelson}
	\CT_{s_r}\circ \ldots\circ \CT_{s_1}(P_0)\{n\}, 
\end{equation} 
where $s_1,\ldots, s_r$ is a sequence in $\widehat{\CS}$ and $n\in \ZZ$. \emph{Bott-Samelson objects} are objects of the form as in \eqref{eq:bottsamelson}.
\begin{definition}
	The full subcategory of $\CK_k$ given by special objects is denoted by $\CM_k$.
\end{definition}

\begin{remark}
We will verify lateron that $Q_0$ appears as a direct summand of $\CT_{s_j}\circ\ldots\circ\CT_{s_1}(P_0)$ where $w_0=s_1\ldots s_j$ is a reduced expression. Since the proof needs further results we will hand it in later (see below, Proposition  \ref{q0summand}). 
Hence we can define a full subcategory $\CM_k^{\circ}\subset\CM_k$ consisting of objects isomorphic to a direct summand of a direct sum of objects of the form
\begin{equation}\label{bottsamlike}
	\CT_{s_r}\circ\ldots\circ\CT_{s_1}(Q_0)\{n\}, 
\end{equation} 
where $s_1,\ldots, s_r$ is a sequence in $\widehat{\CS}$ and $n\in \ZZ$. Objects of the form as in \eqref{bottsamlike} are called \emph{Bott-Samelson like objects}. 
\end{remark}

\subsection{Elementary properties of special objects}\label{presprop}
A common approach to prove a property of special objects in $\CK_k$ is to first show the validity of this property for the unit object $P_0$ and to verify to be preserved by applying the translation functors. If this property is unaffected by taking direct sums and direct summands, we completed.  
This is how one obtains 
\begin{lemma}\label{fingen:torfre}
Let $M$ be a special object. For all $A\in\SA$ and $\beta \in R^{+}$:
\begin{enumerate}
\item[a)] $M(A)$ and $M(A,\beta)$ are finitely generated $S^{\emptyset}_k$- resp. $S^{\beta}_k$-modules.
\item[b)]  $M(A)$ and $M(A,\beta)$ are free $S^{\emptyset}_k$- resp. $S^{\beta}_k$-modules.
\end{enumerate}
\end{lemma}

\begin{lemma} \label{basics:dense}
	Let $M \in \CM_k$, $A \in \SA$ and  $\beta \in R^{+}$. The inclusion 
	\begin{equation*}
\iota_{M(A,\beta)}\colon M(A,\beta) \hookrightarrow M(A) \oplus M (\beta \uparrow A)
	\end{equation*}
induces the following isomorphism of $S_k^{\emptyset}$-modules 	\begin{align}\label{basics:transdense}
		\widehat{\iota}_{M(A,\beta)} \colon M(A,\beta) \otimes_{S_k^{\beta}} S^{\emptyset}_k &\to M(A) \oplus M (\beta \uparrow A)\\
	\nonumber \sum m\otimes \lambda &\mapsto \sum \iota_{M(A,\beta)} (m)\cdot \lambda.
	\end{align}
	It is of degree $0$. 
\end{lemma}

\begin{proof}
This is true for the unit object $P_0$. For all $\beta \in R^+$
\begin{align*}
&\widehat{\iota}_{P_0(A_e,\beta)}\colon  S^{\beta}_k\otimes_{S^{\beta}_k}S^{\emptyset}_k \to S^{\emptyset}_k \text{ resp. }\\
&\widehat{\iota}_{P_0(\beta \downarrow A_e,\beta)}\colon  S^{\beta}_k\otimes_{S^{\beta}_k}S^{\emptyset}_k \to S^{\emptyset}_k 
\end{align*}
is surjective as $\forall s\in S_k^{\emptyset}$ it is $\widehat{\iota}(1\otimes s)=s$.
	
Applying $\CT_s$ inductively yields only objects fulfilling \eqref{basics:transdense}. Assume $M\in \CM_k$ to fulfill \eqref{basics:transdense} and let $s\in \widehat{\CS}$. Then the object $\CT_s M$ fulfills \eqref{basics:transdense}, too: 
Let $\beta \in R^{+}$ and $A \in \SA$ (cf. Figure \ref{fig:cases}). In case $\beta \uparrow \w{A}{s}\neq \w{A}{s}$, the property is obvious with
\begin{equation*}
 \iota_{\CT_s M (A,\beta)} =\iota_{M(\wm{A}{s},\beta)}\oplus\iota_{M(\wp{A}{s},\beta)}.
\end{equation*}

Let $\beta \uparrow \w{A}{s}= \w{A}{s}$, then
\begin{equation*}
 \iota_{\CT_s M (A,\beta)} = \begin{cases} 
 \iota_{M( A,\beta)}\oplus \iota_{M( A,\beta)}& \text{if }A=\wm{A}{s}, \\
 \iota_{M(\beta\downarrow A,\beta)} \oplus  \iota_{M(\beta\uparrow A,\beta)} & \text{if }A=\wp{A}{s}.
\end{cases}\\
\end{equation*}
And furthermore if $\widehat{\iota}_{M(A,\beta)}$ is surjective then 
\begin{equation*}
\widehat{\iota}_{M(A,\beta)}(\beta M(A,\beta) \otimes_{S^{\beta}_k} S^{\emptyset}_k) = M(A)\oplus M(A,\beta),
\end{equation*}
too, since $m\otimes \lambda=\beta m \otimes \beta^{-1}\lambda$ in $M(A,\beta) \otimes_{S^{\beta}_k} S^{\emptyset}_k$. 
\end{proof}

The $S^{\beta}_k$-modules $M(A,\beta)$ and $M(\beta \downarrow A, \beta)$ have in common that both partially map into the $S_k^{\emptyset}$-module $M(A)$. We would like to understand this relation better. Let $M\in \CK_k$, $A\in\SA$ and $\beta \in R^+$. By abuse of notation we will write
\begin{equation*}
M(A,\beta) \cap M(A)= M(A,\beta) \cap [M(A) \oplus 0].
\end{equation*}
For a direct sum $M(A) \oplus M(A')$ we denote the projection to $M(A)$ by $\pr_A$.

\begin{lemma}\label{verma:imagel}
	Let $M \in \CM_k$ be a special object, $A\in \SA$ and $\beta \in R^+$. Then
	\begin{equation}\label{verma:imagedown}
\pr_A [M(A,\beta)]
		=M(\beta \downarrow A, \beta)\cap M(A)
	\end{equation}
	and
	\begin{equation}\label{verma:imageup}
		\beta \, \pr_A [M(\beta \downarrow A,\beta)]
		\subseteq M(A, \beta)\cap M(A).
	\end{equation}
	If $M\in \CM_k^{\circ}$, then the above inclusion is an equality.
\end{lemma}

\begin{proof}
As usual, we firstly examine the object $P_0$. We only need to consider the case when $A=A_e$ because otherwise $P_0(A)$ is the trivial module. One immediately gets 
\begin{equation*}
\pr_{A_e}[P_0(A_e,\beta)]=S^{\beta}_k
		=P_0(\beta \downarrow A_e, \beta)\cap P_0(A_e)
\end{equation*}
and
\begin{equation*}
\beta \, \pr_{A_e} [P_0(\beta \downarrow A_e,\beta)]=\beta  S^{\beta}_k
		\subset S^{\beta}_k = P_0(A_e, \beta)\cap P_0(A_e).		
\end{equation*}
Let now $A$ correspond to an element in the finite Weyl group $\CW$. Then
\begin{align}\label{verma:q0}
\beta\, \pr_A &[Q_0(\beta\downarrow A,\beta)]=\beta S^{\beta}_k\nonumber\\
&=
\begin{cases} \beta S^{\beta}_k \cap S^{\emptyset}_k & \text{if }A\in \beta\uparrow \SA_{\beta}^-, \\
\{(\beta x+y,y)\mid x,y\in S_k^{\beta}\}\cap [S_k^{\emptyset}\oplus 0] & \text{if }A\in \SA_{\beta}^-,
\end{cases}\\
&= Q_0(A,\beta)\cap Q_0(A)\nonumber.
\end{align}

Assume now that an object $M\in \CK_k$ fulfills \eqref{verma:imagedown} for all $\beta \in R^+$ and all $A\in \SA$ and let $s\in \widehat{\CS}$ be an arbitrary simple affine reflection. Then $\CT_s M$ fulfills \eqref{verma:imagedown} as well. We distinguish two cases. In the first case, $\w{A}{s}$ is of type $\beta_1\neq\beta$ and by this $\w{\beta \downarrow A}{s}$ is of a type $\beta_2 \neq \beta$. Then
\begin{equation}\label{verma:case3}
\pr_{A} [\CT_s M(A,\beta)]=\pr_{\wm{A}{s}} [M(\wm{A}{s},\beta)] \oplus \pr_{\wp{A}{s}}  [M(\wp{A}{s},\beta)].
\end{equation}
It is known by induction hypothesis that
\begin{equation*}
\pr_{\wm{A}{s}} [M(\wm{A}{s},\beta)] = M(\beta\downarrow\wm{A}{s},\beta)\cap M(\wm{A}{s})
\end{equation*}
and 
\begin{equation*}
\pr_{\wp{A}{s}}  [M(\wp{A}{s},\beta)]= M(\beta\downarrow\wp{A}{s},\beta)\cap M(\wp{A}{s})
 \end{equation*}
and Lemma \ref{wallcomb}.b) provides
\begin{equation*}
\{\beta\downarrow\wm{A}{s},\beta\downarrow\wp{A}{s}\}=\{\wm{(\beta \downarrow A)}{s}, \wp{(\beta \downarrow A)}{s}\}.
\end{equation*}  
Hence, \eqref{verma:case3} equals
\begin{equation*}
\CT_s M(\beta \downarrow A,\beta)\cap \CT_s M(A).
\end{equation*}
In the second case, both $\w{A}{s}$ and $\w{(\beta \downarrow A)}{s}$ are of type $\beta$. If $A$ is a \emph{negative alcove with respect to $s$}, that means $A=\wm{A}{s}$, $\beta \downarrow A$ is positive with respect to $s$ and 
\begin{equation*}
\pr_{A} [\CT_s M (A, \beta)] = M(A,\beta)
\end{equation*}
is equal to
\begin{equation*}
\CT_s M (\beta \downarrow A, \beta)\cap \CT_s M(A) =M(A,\beta).
\end{equation*}
If $A$ is a positive alcove, i.e. , $\beta \downarrow A$ is negative and thereby
\begin{equation*}
\pr_A[\CT_s M (A, \beta)]=\beta  M(\beta \downarrow A,\beta) 
\end{equation*}
while
\begin{equation*}
\{(\beta x+y,y)\mid x,y \in M(\beta \downarrow A,\beta)\}\cap [0\oplus (\CT_s M)(A)]=\beta M(\beta \downarrow A,\beta).
\end{equation*}

If $M\in \CK_k$ fulfills the condition in \eqref{verma:imageup} for all $\beta \in R^+$ and all $A\in \SA$, then (\ref{verma:imageup}) is also true for the object $\CT_s M$ where $s\in \widehat{\CS}$ is an arbitrary simple affine reflection. To prove this statement one calculates
\begin{align*}
	\beta \, \pr_A &[\CT_s M(\beta \downarrow A,\beta)]\\
	&=\beta
		\begin{cases}
			M(\beta\uparrow(\beta\downarrow A),\beta)=M(A,\beta)&\text{if }\beta \uparrow \w{A}{s}=\w{A}{s}\\
			&\quad\text{and }A=\wm{A}{s}\\
			M(\beta\downarrow A, \beta)&\text{if }\beta \uparrow\w{A}{s}=\w{A}{s}\\
			&\quad\text{and }A=\wp{A}{s}.\\
	\end{cases}
\end{align*}
This immediately coincides with
\begin{equation*}
\CT_s M(A,\beta) \cap \CT_s M(A)
\end{equation*}
for $\beta \uparrow \w{A}{s}=\w{A}{s}$.
For $\beta \uparrow \w{A}{s}\neq \w{A}{s}$ it is
\begin{align*}
\beta \, &\pr_A [\CT_s M(\beta \downarrow A,\beta)]\\
&=
\beta \, \pr_{\wm{A}{s}}[M(\beta \downarrow\wm{A}{s},\beta)] \oplus \beta \, \pr_{\wp{A}{s}}[M(\beta \downarrow\wp{A}{s},\beta)].
\end{align*} 
But by induction hypothesis this is included in
\begin{equation*}
M(\wm{A}{s},\beta)\cap M(\wm{A}{s})\oplus M(\wp{A}{s},\beta)\cap M(\wp{A}{s})=\CT_s M(A,\beta) \cap \CT_s M(A).
\end{equation*}
If the induction base is $Q_0$, we even obtain equality as in \eqref{verma:q0}.

Both properties are also preserved by taking direct sums, direct summands and by shifting.
\end{proof}

\begin{corollary}\label{projtimesbeta}
Let $(m_1,m_2) \in M(A)\oplus M(\beta \uparrow A)$ such that $(m_1, m_2)\in M(A,\beta)$. Then $(\beta m_1, 0) \in M(A,\beta) \text{ and } (0, \beta m_2) \in M(A,\beta)$.
\end{corollary}

\section{The Duality}
\begin{definition}
Let $M \in \CM_k$ with 
\begin{equation*}
M=(\{M(A)\}_{A\in \SA}, \{M(A,\beta)\}_{A\in \SA, \beta \in R^{+}}).
\end{equation*}
 We define the  \emph{duality $\cD$} in the following way:
	\begin{align*}
		\cD M (A) &= \Hom^{\bullet}_{S^{\emptyset}_k}(M(A),S^{\emptyset}_k),\\
		\cD M (A,\beta) &= \Hom^{\bullet}_{S^{\beta}_k}(M(A,\beta),S^{\beta}_k).
	\end{align*}		
\end{definition}
Let $B$ and $C$ be two graded $S^{\emptyset}_k$-modules. A graded homomorphism $\varphi\colon B \to C$ is of degree $l$ if
\begin{equation*}
	\varphi(B_n) \subseteq C_{l+n}
\end{equation*}
for all $n\in \ZZ$. 
The module of all homomorphisms of degree $l$ is denoted by 
\begin{equation*}
		\Hom_{S^{\emptyset}_k}^l (B,C).
\end{equation*}

Following Lemma \ref{fingen:torfre} we only consider finitely generated $S^{\emptyset}_k$-modules, and hence each module $\cD M(A)$ carries a natural $\ZZ$-grading, namely
\begin{equation*}
	\cD M(A)_l= \Hom_{S^{\emptyset}_k}^l (M(A),S^{\emptyset}_k).
\end{equation*}

In the following, we want to describe how $\cD M(A,\beta)$ embeds into $\cD M(A)\oplus \cD M(\beta \uparrow A)$. As each $M(A,\beta)$ is torsion free as seen in Lemma \ref{fingen:torfre},  there is the embedding
	\begin{equation*}
	-\otimes_{S^{\beta}_k}1     \colon	\Hom_{S^{\beta}_k}(M(A,\beta),S^{\beta}_k) \hookrightarrow \Hom_{S^{\beta}_k}(M(A,\beta),S^{\beta}_k)\otimes_{S^{\beta}_k} S^{\emptyset}_k
	\end{equation*}
	where the right-hand side is isomorphic to 
	\begin{equation*}
		\Hom_{S^{\beta}_k}(M(A,\beta), S^{\emptyset}_k)
	\end{equation*}
	as $M(A,\beta)$ is finitely generated as an $S_k^{\beta}$-module. This yields an isomorphism to 
	\begin{equation}\label{rightside}
		\Hom_{S^{\emptyset}_k}(M(A,\beta)\otimes_{S^{\beta}_k}S^{\emptyset}_k, S^{\emptyset}_k).
	\end{equation}
	But $M \in \CM_k$ and hence by Lemma \ref{basics:dense} 
	\begin{equation*}
		M(A,\beta)\otimes_{S^{\beta}_k}S^{\emptyset}_k \cong M(A) \oplus M(\beta \uparrow A).
	\end{equation*}  
	So the module in  \eqref{rightside} in turn is isomorphic to 
	\begin{equation*}
		\Hom_{S^{\emptyset}_k}(M(A) \oplus M (\beta \uparrow A),S^{\emptyset}_k).
	\end{equation*}
$\varphi \in \cD M(A,\beta)$ and $\varphi \otimes 1 \in \cD M(A)\oplus \cD M(\beta \uparrow A)$ coincide on $M(A,\beta)$. 
All homomorphisms above are of degree 0 so $\cD M(A,\beta)$ embeds gradedly into $\cD M(A) \oplus \cD M(\beta\uparrow A)$ and by this it is verified that for a special object $M \in \CM_k$ the dual object $\cD M$ is indeed an object in $\CK_k$. 

In addition the following equations hold for graded  $S^{\emptyset}_k$-modules $B$ and $C$:
\begin{equation*}
	\Hom_{S^{\emptyset}_k}^l (B,C)=\Hom_{S^{\emptyset}_k}^0 (B\{l\},C)=\Hom_{S^{\emptyset}_k}^0 (B,C\{-l\}), 
\end{equation*}
where the same can be done for $S^{\beta}_k$-modules. 
Hence, the functors
\begin{equation}\label{dualshift}
	\cD \circ \{n\} \simeq \{-n\} \circ \cD
\end{equation}
between $\CM_k$ and $\CK_k$ are isomorphic.

\subsection{The relation between $\cD$ and $\CT_s$}

To understand the duality $\cD$ in $\CK_k$ better, we need to know the relation between the translation functors $\CT_s$ and $\cD$ besides the already seen relation of $\cD$ to $\{n\}$ in \eqref{dualshift}.

To do so, we need to introduce a further set of \emph{dual} translation functors $\{\CT^{\vee}_{s}\mid s\in \widehat{\CS}\}$ which only differs to $\{\CT_{s}\mid s\in \widehat{\CS}\}$ by interchanging the role of up-arrows and down-arrows. More precisely, this affects the case when $\beta \uparrow \w{A}{s}=\w{A}{s}$ and $A=\wp{A}{s}$ and is realised by instead of shifting $M(\beta \downarrow A,\beta)$ to shift $M(\beta\uparrow A,\beta)$  in degree. Concretely:
 
\begin{align*}
\CT^{\vee}_s M (A) &= M(\wm{A}{s}) \oplus M(\wp{A}{s}) \text{ and }\\
\CT^{\vee}_s M (A,\beta) & =  
\begin{cases} 
	\{(\beta x + y,y)\mid x,y \in M(A,\beta)\} & \text{if $\beta \uparrow \w{A}{s}=\w{A}{s}$}\\
				&\text{ and $A=\wm{A}{s}$,}\\
	M(\beta \downarrow A,\beta) \oplus \beta  M(\beta \uparrow A,\beta) & \text{if $\beta \uparrow \w{A}{s}=\w{A}{s}$}\\
				&\text{ and $A=\wp{A}{s}$,}\\
	M(\wm{A}{s},\beta) \oplus M(\wp{A}{s},\beta)&\text{if }\beta \uparrow \w{A}{s}\neq\w{A}{s}.
\end{cases}
\end{align*}
A lot of properties preserved by $\CT_s$ are preserved by $\CT^{\vee}_s$, too. This is in particular true for the properties in the Lemmas \ref{fingen:torfre} and \ref{basics:dense}.
\begin{proposition}\label{theo:dualtrans}
Let $s \in \widehat{\CS}$. Then the functor 
\begin{equation*}
\cD \circ \CT_{s}\colon \CM_k \to \CK_k
\end{equation*}
is isomorphic to 
\begin{equation*}
\{-2\}\circ \CT^{\vee}_{s} \circ \cD \colon \CM_k \to \CK_k.
\end{equation*}
\end{proposition}

\begin{proof}
We will give the natural transformations between 
\begin{equation*}
\cD \circ \CT_{s}\overset{\tau}{\underset{\sigma}{\rightleftarrows}} \{-2\}\circ\CT^{\vee}_{s} \circ \cD
\end{equation*}
explicitly. For $A\in \SA$ define $\alpha_s (A):=\sign(A) \alpha(\w{A}{s})$ where $\alpha(\w{A}{s})\in R^+$ is the type of $\w{A}{s}$ (i.e. $\alpha(\w{A}{s}) \uparrow \w{A}{s}=\w{A}{s}$) and 
\begin{equation*}
\sign(A)=
\begin{cases}
1 &\text{ if }A=\wp{A}{s},\\
-1 & \text{ if }A=\wm{A}{s}.
\end{cases}
\end{equation*}
Keep in mind that
\begin{align*}
[\cD \circ \CT_{s}] M(A) &= \Hom_{S_k^{\emptyset}}(\CT_s M(A), S_k^{\emptyset})\\
&= \Hom_{S_k^{\emptyset}}(M(\wm{A}{s}) \oplus M(\wp{A}{s}), S_k^{\emptyset})
\end{align*}
and 
\begin{align*}
[\cD \circ \CT_{s}] M(A,\beta) &= \Hom_{S_k^{\beta}}(\CT_s M(A,\beta), S_k^{\beta})\\
&= \begin{cases}
&\Hom_{S_k^{\beta}}(\{\beta x+y,y)\mid x,y \in M(A,\beta)\}, S_k^{\beta})\\
	&\qquad\qquad\text{if $\beta \uparrow \w{A}{s}=\w{A}{s}$ and $A=\wm{A}{s}$,}\\
&\Hom_{S_k^{\beta}}(\beta M(\beta \downarrow A,\beta)\oplus M(\beta \uparrow A, \beta), S_k^{\beta}) \\
				&\qquad\qquad\text{if $\beta \uparrow \w{A}{s}=\w{A}{s}$ and $A=\wp{A}{s}$,}\\
&\Hom_{S_k^{\beta}}(M(\w{A}{s}_-,\beta) \oplus M(\wp{A}{s},\beta), S_k^{\beta})\\
&\qquad\qquad\text{if } \beta \uparrow \w{A}{s}\neq\w{A}{s}.
\end{cases}
\end{align*} 

Similarly, we obtain
\begin{align*}
[\{-2\} \circ \CT_{s}^{\vee}\circ \cD] M(A) =\cD M(\wm{A}{s})\{-2\}\oplus \cD M(\wp{A}{s})\{-2\}\\
\qquad \qquad= \Hom_{S_k^{\emptyset}}(M(\wm{A}{s}), S_k^{\emptyset})\{-2\}\oplus \Hom_{S_k^{\emptyset}}(M(\wp{A}{s}), S_k^{\emptyset})\{-2\}\\
\end{align*}
and 
\begin{align*}
&[\{-2\} \circ \CT_{s}^{\vee}\circ \cD] M(A,\beta) \\&= \begin{cases}
&\{\beta \varphi+\psi,\psi)\mid \varphi, \psi \in \cD M(A,\beta)\}\{-2\}\\
&\qquad\qquad\text{if } \beta \uparrow \w{A}{s}=\w{A}{s}, A=\wm{A}{s},\\
&\cD M(\beta \downarrow A,\beta)\{-2\}\oplus \beta \cD M(\beta \uparrow A, \beta)\{-2\}\\
&\qquad\qquad\text{if } \beta \uparrow \w{A}{s}=\w{A}{s}, A=\wp{A}{s},\\
&\cD M(\wm{A}{s},\beta)\{-2\} \oplus \cD M(\wp{A}{s},\beta)\{-2\}\\
&\qquad\qquad\text{if } \beta \uparrow \w{A}{s}\neq\w{A}{s}.
\end{cases}
\end{align*} 

For any special object $M$ define the natural transformation $\tau^M=(\tau^M_A)_{A\in\SA}$ resp. $\sigma^M=(\sigma^M_A)_{A\in\SA}$ as follows.
We set 
\begin{align*}
&\tau^M_A\colon &[\cD \circ \CT_{s}] M(A) &\to [\{-2\}\circ\CT^{\vee}_{s} \circ \cD ]M(A),\\
&&\varphi &\mapsto \alpha_s(A)(\varphi|_{M(\wm{A}{s})}\oplus \varphi|_{M(\wp{A}{s})}),\\
&\text{ and }&&\\
&\sigma^M_A\colon &[\{-2\}\circ\CT^{\vee}_{s} \circ \cD ]M(A)&\to [\cD \circ \CT_{s}]M(A),\\
&&(\varphi_{M(\wm{A}{s})},\varphi_{M(\w{A}{s}_+)}) &\mapsto \alpha_s(A)^{-1}(\varphi_{M(\w{A}{s}_-)} +  \varphi_{M(\w{A}{s}_+)}).
\end{align*}

Each $\tau^M$ is of total degree $2$ and each $\sigma^M$ is of total degree $-2$ and hence after the $\{-2\}$-shift of $\CT_s^{\vee}\circ \cD$ both are of degree $0$. Under the assumption that $\tau$ and $\sigma$ indeed are natural transformations and by the fact that  $\tau^M_A\circ \sigma^M_A=\id_{[\{-2\}\circ\CT_s^{\vee}\circ \cD]M(A)}$ and $\sigma^M_A\circ \tau^M_A=\id_{[\cD\circ \CT_s]M(A)}$ for each $A\in \SA$ and $M \in \CM_k$, we get the claim. 

In order to complete the proof the following must be verified:
\begin{align*}
&\tau^M_A\oplus \tau^M_{\beta \uparrow A} ([\cD \circ \CT_{s}]M(A,\beta))\subseteq [\{-2\}\circ\CT^{\vee}_{s} \circ \cD]M(A,\beta) \quad\text{ and }\\
&\sigma^M_A\oplus \sigma^M_{\beta \uparrow A} ([\{-2\}\circ\CT^{\vee}_{s} \circ \cD ]M(A,\beta))\subseteq [\cD \circ \CT_{s} ]M(A,\beta)
\end{align*}
for any $\beta \in R^{+}$ and $A \in \SA$.

Consider the case when $\beta \uparrow \w{A}{s} =\w{A}{s}$ and $A=\wm{A}{s}$ (cf. Figure \ref{fig:cases}). In this case, we get $\alpha_s(A)=-\beta$ and $\alpha_s(\beta \uparrow A)=\beta$. $M(A,\beta)$ is finitely generated by $g_1, \ldots, g_j$ with the dual homomorphisms $g_1^*,\ldots, g_j^*$ as it is free over $S_k^{\beta}$. Then the generators of $\CT_s M (A,\beta)$ are 
\begin{equation*}
(g_1,g_1),(\beta g_1,0),\ldots,(g_j,g_j),(\beta g_j,0). 
\end{equation*}
Hence 
\begin{equation*}
(0,g_1^*),(\beta^{-1}g_1^*,-\beta^{-1}g_1^*),\ldots,(0,g_j^*),(\beta^{-1}g_j^*,-\beta^{-1}g_j^*)
\end{equation*}
 generate $[\cD \circ \CT_s] M(A,\beta)$ which are mapped in turn by $\tau^M_A\oplus\tau^M_{\beta \uparrow A}$ to  
\begin{equation*}
(0,\beta g_1^*),(-g_1^*,-g_1^*),\ldots,(0,\beta g_j^*), (-g_j^*,-g_j^*).
\end{equation*}
But this is a generating set of $[\{-2\}\circ\CT_s^{\vee}\circ \cD]M(A,\beta)$. 

Next, we study the case $\beta \uparrow \w{A}{s}=\w{A}{s},A=\wp{A}{s}$. Here, $\alpha_s(A)=\beta$ and $\alpha_s(\beta \uparrow A)=-\beta$. In this case $[\cD \circ \CT_{s}]M (A,\beta)$ is 
\begin{equation*}
	\Hom_{S^{\beta}_k}(\beta M(\beta \downarrow A,\beta), S^{\beta}_k)\oplus\Hom_{S^{\beta}_k}(M(\beta \uparrow A,\beta), S^{\beta}_k).
\end{equation*}
For an element $\varphi \in [\cD\circ \CT_{s}]M (A,\beta)$ we write $\varphi=(\varphi_1,\varphi_2)$ with $\varphi_1\in \cD(\beta M(\beta \downarrow A,\beta))$ and $\varphi_2 \in \cD(M(\beta \uparrow A,\beta))$. By $\tilde{\varphi_1}\colon M(\beta \downarrow A,\beta)\to S^{\beta}_k$ we denote the homomorphism $\tilde{\varphi_1}(m)=\varphi_1(\beta m)$. Thus
\begin{equation*}
\varphi=(\varphi_1,\varphi_2)\hookrightarrow (\tilde{\varphi_1}\otimes \beta^{-1},\varphi_2 \otimes 1)\overset{\tau}{\mapsto}(\tilde{\varphi_1}\otimes 1,-\varphi_2\otimes \beta), 
\end{equation*}
but this corresponds to the element $(\tilde{\varphi_1},\beta \varphi_2)$ in $[\{-2\}\circ\CT^{\vee}_s \circ \cD ]M(A,\beta)$. In both cases, 
the opposite direction via $\sigma$ works analogously. 

The residual part is $\beta \uparrow \w{A}{s}\neq \w{A}{s}$. In this particular case, $\alpha_s(A)\notin \{-\beta, \beta\}$, but at least $\alpha_s(\beta \uparrow A) =\alpha_s(A) + \beta z$, $z \in \ZZ$. Let $\varphi \in \cD(M(\w{A}{s}_{\pm},\beta))$ and $(m_1,m_2) \in M(\w{A}{s}_{\pm},\beta)$. With Corollary \ref{projtimesbeta} we get $(0,\beta m_2) \in M(\w{A}{s}_{\pm},\beta)$ and then
\begin{align*}
&\tau_A\oplus \tau_{\beta \uparrow A}(\varphi \otimes 1)(m_1, m_2)\\
&\qquad = \alpha_s(A) (\varphi \otimes 1)(m_1,m_2) + z (\varphi \otimes 1)(0, \beta m_2) \in S^{\beta}_k.
\end{align*}
Hence $\tau_A\oplus \tau_{\beta \uparrow}(\varphi \otimes 1)\in \cD(M(\w{A}{s}_{\pm},\beta))$. 
Because of $\alpha_s(A)\neq \pm \beta \neq \alpha_s(\beta \uparrow A)$, the elements $\alpha_s(A)$ and $\alpha_s(\beta \uparrow A)$ are not only units in $S^{\emptyset}_k$ but as well in $S^{\beta}_k$ and therefore, multiplication by $\alpha_s(A)^{-1}$ resp. $\alpha_s(\beta \uparrow A)^{-1}$ means multiplication by an element in $S^{\beta}_k$. Hence, 
\begin{eqnarray*}
\sigma_A\oplus \sigma_{\beta\uparrow A}(\varphi \otimes 1)(m_1, m_2)\\
= \alpha_s(A)^{-1}(\varphi \otimes 1)(m_1,m_2) - z(\alpha_s(\beta \uparrow A))^{-1} (\varphi \otimes 1)(0, \beta m_2) \in S^{\beta}_k 
\end{eqnarray*}
and $\sigma_A\oplus \sigma_{\beta \uparrow A}(\varphi \otimes 1)\in \cD(M(\w{A}{s}_{\pm},\beta))$. 
\end{proof}

\subsection{The tilting functor $\CC_{w_0}$}
So far there is the following situation. $\{\CT_{s}\mid s\in \widehat{\CS}\}$ is the set of translation functors on $\CK_k$ which spans a subcategory $\CM_k$. Equally, the set $\{\CT^{\vee}_{s}\mid s\in \widehat{\CS}\}$ of the so-called dual translation functors spans a subcategory $\CM^{\vee}_k$ of similar type. The duality interchanges objects of those two subcategories. 

Our goal is to relate those two categories and to find a concept of self-duality of special objects. 
$\CT_s$ and $\CT^{\vee}_s$ differ essentially in the order arising from $\uparrow$ resp. $\downarrow$. This gives reason to define a covariant functor that tilts this order. 

Let $w\in \CW$. We want to distinguish the two situations where $w(\beta)\in R^{+}$ and $w(\beta) \in R^{-}$ for a positive root $\beta$. Therefore, recall the definition of
\begin{equation*}
w(\beta)^+=\begin{cases}
  w(\beta)  & \text{if }w(\beta) \in R^{+}, \\
  -w(\beta) & \text{if }w(\beta) \in R^{-}.
\end{cases}
\end{equation*}
Furthermore, we will need to move from $S^{w(\beta)^+}_k$-modules to $S^{\beta}_k$-modules. This is done by the twist functor 
\begin{equation*}
t_w \colon S^{w(\beta)^+}_k\catmod \to S^{\beta}_k\catmod,
\end{equation*}
where the $S^{\beta}_k$-action on a module $t_w(M)$ is given by
\begin{equation*}
\forall s\in S^{\beta}_k,\, m\in t_w(M)=M: \ s.m=w^{-1}(s).m.
\end{equation*}
Here $w^{-1}$ is extended to $S^{\emptyset}_k$ with
\begin{align*}
w^{-1}\colon S^{\emptyset}_k &\to S^{\emptyset}_k\\
1 &\mapsto 1, \\
\alpha &\mapsto w^{-1}( \alpha).
\end{align*}

Let $N$ be an object in $\CK_k$. The following defines a functor $\CC_{w}\colon \CK_k \to \CK_k$ depending on an element $w\in \CW$:
\begin{align*}
	\CC_{w} N (A) &=N(w.A),\\
	\CC_{w} N (A,\beta)&=
		\begin{cases}
	t_{w} N(w.(\beta \uparrow A), w(\beta)^+) & \text{if }w(\beta) \in R^{-}\\
	t_{w} N(w.A, w(\beta)^+)& \text{if }w(\beta) \in R^{+}.
		\end{cases}
\end{align*}
We will call $\CC_{w}$ \emph{tilting functor on $w$}. 

The next step is to verify that $\CC_{w}$ is well-defined. If $N$ is an object in $\CK_k$, we need to make sure that $\CC_w N (A,\beta)$ is an $S^{\beta}_k$-submodule of $\CC_w N (A) \oplus \CC_w N (\beta\uparrow A)$ for each $\beta \in R^{+}, A \in \SA$. Let $w(\beta)\in R^{+}$. By definition
\begin{equation*}
	\CC_w N (A,\beta)=t_w N(w.A,w(\beta)^+).
\end{equation*}
As $N\in \CK_k$, $N(w.A,w(\beta)^+)$ embeds into $N(w.A)\oplus N(w(\beta)^+\uparrow w.A)$.
By Lemma \ref{kipp:wb} the previous term equals
\begin{equation*}
	N(w.A)\oplus N(w.(\beta\uparrow A)) = \CC_w N (A) \oplus \CC_w N(\beta \uparrow A). 
\end{equation*}
The case where $w(\beta) \in R^{-}$ works similarly. 
\begin{equation*}
	\CC_w N (A,\beta)=t_w N(w.(\beta \uparrow A),w(\beta)^+)
\end{equation*}
embeds into 
\begin{equation*}
	N(w.(\beta \uparrow A))\oplus N(w(\beta)^+\uparrow w.(\beta \uparrow A))
\end{equation*}
 and by Lemma \ref{kipp:wb} this equals
\begin{equation*}
	\CC_w N(\beta \uparrow A) \oplus \CC_w N(A).
\end{equation*}
Please notice that in this case the role of $A$ and $\beta \uparrow A$ were exchanged.

The tilting functor would not be a functor in general if it was defined the same way for the case $w(\beta) \in R^{-}$ as for the case $w(\beta) \in R^{+}$. As we have seen this would violate the condition 
	\begin{equation*}
		\CC_w N (A,\beta)\hookrightarrow\CC_w N (A) \oplus \CC_w N 	
		(\beta\uparrow A)
	\end{equation*}
	for some objects in $\CK_k$. 

From now on, we will restrict to the case $w=w_0$ where $w_0$ is the longest element of the finite Weyl group. We will denote $\CC_{w_0}$ simply by $\CC$. Please keep in mind that this implies $w(\beta) \in R^{-}$ for all $\beta \in R^+$.

\begin{proposition}\label{theo:kipptrans}
Let $s \in \widehat{\CS}$. Then the functors 
\begin{equation*}
	\CC \circ \CT^{\vee}_s \colon \CK_k \to \CK_k
\end{equation*}
and 
\begin{equation*}
	\CT_s \circ \CC \colon \CK_k \to \CK_k
\end{equation*}
are isomorphic. 
\end{proposition}

\begin{proof}
	Let $N$ be an object in $\CK_k$ and we will use $w=w_0$ for convenience. By definition 
	\begin{align*}
		[\CC \circ \CT^{\vee}_s] N (A)&= N(\wm{(w.A)}{s})\oplus 
			N(\wp{(w.A)}{s})\text{ and }\\
		[\CT_s\circ \CC ] N (A)&=N(w.(\wm{A}{s}))\oplus N(w.(\wp{A}{s})).
	\end{align*}
By Lemma \ref{kipp:arithmetik}.a) these modules coincide. Let us now consider 
\begin{equation*}
	[\CC \circ\CT_s^{\vee}]N(A,\beta)=t_w \CT_s^{\vee} N(w.(\beta\uparrow A),w(\beta)^+).
\end{equation*}
As in Remark \ref{updown} and with Lemma \ref{kipp:arithmetik}.b) it is equivalent
\begin{align*}
	\beta \uparrow \w{A}{s}=\w{A}{s} &\iff \beta \uparrow\w{(\beta \uparrow A)}{s}=\w{(\beta \uparrow A)}{s}\\
	& \iff w(\beta)^+\uparrow \w{(w.(\beta\uparrow A))}{s}=\w{(w.(\beta\uparrow A))}{s}.
\end{align*}
Furthermore, let $\beta \uparrow \w{A}{s}=\w{A}{s}$ and $A=\wm{A}{s}$. We know that
\begin{equation*}
\w{(w.(\beta\uparrow A))}{s}=w.(\w{(\beta\uparrow A)}{s})=w.(\w{A}{s}),
\end{equation*}
where the second equation follows from Lemma \ref{wallcomb}. Hence
\begin{align*}
\{\wm{(w.(\beta\uparrow A))}{s},\wp{(w.(\beta\uparrow A))}{s}\}&=\{w.(\beta\uparrow A),w.A\}\\
&=\{w(\beta)^+\downarrow w.A, w.A\},
\end{align*}
which finally leads to the necessary statement
\begin{equation*}
A=\wm{A}{s}\iff w.(\beta \uparrow A)=\wm{(w.(\beta\uparrow A))}{s}.
\end{equation*}

	\ul{Case 1} ($\beta \uparrow \w{A}{s}=\w{A}{s}, A=\wm{A}{s}$, cf. Figure \ref{fig:cases}):
			By the preceding we obtain 
			\begin{equation*}
			[\CC\circ \CT_s^{\vee}] N (A,\beta)=t_w  \{(w(\beta)^+x + y, y)\mid x,y \in N(w.(\beta \uparrow A),w(\beta)^+)\}
			\end{equation*}
			which coincides with 
			\begin{equation*}
			[\CT_s \circ\CC] N(A,\beta)=\{(\beta x + y,y)\mid x,y\in t_w N(w.(\beta \uparrow A),w(\beta)^+\}.
			\end{equation*}							
						
	\ul{Case 2} ($\beta \uparrow \w{A}{s}=\w{A}{s}, A=\wp{A}{s}$):
			By definition
			\begin{align*}
				&[\CC\circ \CT^{\vee}_s] N (A,\beta)\\
				&=t_w [ N(w(\beta)^+\downarrow w.(\beta \uparrow A),w(\beta)^+)\oplus w(\beta)^+ N(w(\beta)^+ \uparrow w.(\beta \uparrow A), w(\beta)^+)],
			\end{align*}			
			while
			\begin{equation*}
				[\CT_s \circ\CC] N (A,\beta)=\beta t_w  N(w.A,w(\beta)^+)\oplus t_w N(w.(\beta \uparrow^2 A),w(\beta)^+).
			\end{equation*}
			But this coincides because with Lemma \ref{kipp:wb}
			\begin{align*}
				w(\beta)^+\downarrow w.(\beta \uparrow A)&=w.(\beta \uparrow^2 A)	 \text{ and }\\
				w(\beta)^+\uparrow w.(\beta\uparrow A)&=w.A.		
			\end{align*}
			
	\ul{Case 3} ($\beta \uparrow \w{A}{s}\neq\w{A}{s}$):
	\begin{equation*}
				[\CC \circ\CT^{\vee}_s] N (A,\beta)=t_w [ N(\wm{(w.(\beta \uparrow A))}{s},w(\beta)^+)\oplus N(\wp{(w.(\beta \uparrow A))}{s}, w(\beta)^+)],
			\end{equation*}
			while
			\begin{equation*}
				[\CT_s \circ\CC] N (A,\beta)= t_w N(w.(\beta \uparrow \wm{A}{s}),w(\beta)^+)\oplus t_w N(w.(\beta \uparrow \wp{A}{s}),w(\beta)^+).
			\end{equation*}
			Because of Lemma \ref{kipp:arithmetik}.a) 
			\begin{equation*}
				\{\wm{(w.(\beta\uparrow A))}{s}, \wp{(w.(\beta\uparrow A))}{s}\}=\{w.(\wm{(\beta\uparrow A)}{s}),w.(\wp{(\beta\uparrow A)}{s})\}
			\end{equation*}
			and it is enough to see that
			\begin{equation*}
				\{\wm{(\beta\uparrow A)}{s},\wp{(\beta\uparrow A)}{s}\}=\{\beta \uparrow\wm{A}{s},\beta \uparrow\wp{A}{s}\}.				
			\end{equation*}
			This is true as it was observed in  Lemma \ref{wallcomb}.
\end{proof}

This is not true for general $w\in \CW$. For $w\neq w_0$ we cannot exclude the case where $w(\beta)\in R^{+}$. But this means a violation in the case where $\beta \uparrow \w{A}{s}=\w{A}{s}$, $A=\wp{A}{s}$ which means in formulas
\begin{equation*}
[\CC \circ\CT^{\vee}_s] N (A,\beta)= t_w N(w.(\beta\downarrow A),w(\beta)^+)\oplus \beta t_wN(w.(\beta \uparrow A),w(\beta)^+)
\end{equation*}
and on the other side
\begin{equation*}
[\CT_s\circ\CC] N (A,\beta)=\beta t_w N(w.(\beta\downarrow A),w(\beta)^+)\oplus t_w N(w.(\beta \uparrow A),w(\beta)^+).
\end{equation*}

\section{Some tilting self-dual objects in $\CM_k$}

\subsection{The object $Q_0$}
Properties of objects in $\CM_k^{\circ}$ certainly depend on the object $Q_0$ from where the translation functors span $\CM_k^{\circ}$. This gives rise to the study of $Q_0$.
\begin{lemma}
	$Q_0$ is indecomposable. 
\end{lemma}
\begin{proof}
	Assume that $Q_0$ was decomposable, i.e. $Q_0=V\oplus W$ with $V\neq 0\neq W$. 
	We define the support of an object $M$ as follows:	
	\begin{equation*}
		\supp\,M=\{A \in \SA \mid M(A) \neq 0\}.
	\end{equation*}
	In our case we obtain 
	\begin{equation*}
		\CW = \supp\, Q_0= \supp\, V\,\dot\cup\, \supp\, W,
	\end{equation*}
	as the rank of each $Q_0 (A)$ is $\leq 1$. For both $V$ and $W$ the support is non-empty. But for any $w \in \supp\, V$ (resp. $W$) follows $s_{\beta}w \in \supp\, V$ (resp. $W$) because either $Q_0 (A,\beta)$ or $Q_0(\beta \uparrow A, \beta)$ equals to $\{(\beta x+y,y)\mid x,y \in S^{\beta}_k\}$. An object $M$ \emph{links} the alcoves $A$ and $\beta\uparrow A$ if $M(A,\beta)\hookrightarrow M(A)\oplus M(\beta \uparrow A)$ does not decompose into $S_k^{\beta}$-submodules of $M(A)\oplus 0$ and $0\oplus M(\beta\uparrow A)$. Hence $Q_0$ links $A$ and $\beta \uparrow A$. As $\rk\, Q_0 (A)=\rk\, Q_0(\beta \uparrow A)=1$ this implies that $A$ and $\beta \uparrow A$ lie in the support of the same indecomposable object. 
Take now $w_0 \in \supp\, V$ and via induction over all $\beta \in R^{+}$ we receive 
	\begin{equation*}	
		\CW=\supp\, V.
	\end{equation*}
	This contradicts the decomposition of $Q_0$.	 
\end{proof}

The next step is to compute $[\CC\circ \cD]Q_0$. In the $S^{\emptyset}_k$-components it is easily seen to be
\begin{align*}
[\CC\circ \cD] Q_0(A)&=
		\begin{cases}
	\Hom_{S^{\emptyset}_k}(S^{\emptyset}_k,S^{\emptyset}_k)& \text{if }A = A_w \text{ for } w\in \CW,\\
	0&\text{else}
		\end{cases}
\end{align*}
because $w_0 w \in \CW \iff w \in \CW$.
For any $\beta \in R^{+}$ and $A \in \SA$
\begin{align*}
[\CC\circ \cD] Q_0(A,\beta)&=
		\begin{cases}
	\Hom_{S^{\beta}_k}(\{(\beta x+y,y)\mid x,y \in S^{\beta}_k\},S^{\beta}_k)\\
	 \qquad\qquad\qquad\text{if }w_0.(\beta\uparrow A)\in \SA_{-}^{w(\beta)^+},\\
	\Hom_{S^{\beta}_k}(\beta S^{\beta}_k,S^{\beta}_k)\\
	\qquad\qquad\qquad \text{if }w_0.(\beta\uparrow A)\in w(\beta)^+ \uparrow \SA_{-}^{w(\beta)^+},\\
	\Hom_{S^{\beta}_k}(S^{\beta}_k,S^{\beta}_k)\\
	 \qquad\qquad\qquad\text{if }w_0.(\beta\uparrow A)\in w(\beta)^+ \downarrow \SA_{-}^{w(\beta)^+},\\
		0\qquad\qquad\qquad\qquad\qquad\text{else}.
		\end{cases}
\end{align*}

\begin{remark}
\begin{enumerate}
\item[a)] $w_0.(\beta\uparrow A) \in \SA_{-}^{w(\beta)^+}$ is equivalent to $A \in \SA_{-}^{\beta}$. That is because of the following arguments' chain: 
\begin{align*}
w_0.(\beta \uparrow A)&= w(\beta)^+ \downarrow w_0 .A \in \SA_{-}^{w(\beta)^+} \\
&\iff w_0 .A \in w(\beta)^+ \uparrow \SA_{-}^{w(\beta)^+} \\
&\iff A \in \beta \downarrow \SA_{+}^{\beta} \\
&\iff A\in \SA_{-}^{\beta},
\end{align*}
where we used Lemma \ref{kipp:wb}.
\item[b)]  The equivalence of $w_0.(\beta\uparrow A) \in w(\beta)^+ \uparrow \SA_{-}^{w(\beta)^+}$ and $A \in \beta \downarrow\SA_{-}^{\beta}$ is received by 
using part a) and replacing $A$ by $\beta \uparrow A$.
\end{enumerate}
\end{remark}
By this remark, $[\CC\circ \cD] Q_0$ transforms in the $\beta$-components into 
\begin{align*}
[\CC\circ \cD]Q_0(A,\beta)&=
		\begin{cases}
	\Hom_{S^{\beta}_k}(\{(\beta x+y,y)\mid x,y \in S^{\beta}_k\},S^{\beta}_k)\\
	\qquad\qquad\qquad\qquad\qquad
	 \text{if }A\in \SA_{-}^{\beta},\\
	\Hom_{S^{\beta}_k}(S^{\beta}_k,S^{\beta}_k)\\
	\qquad\qquad\qquad\qquad\qquad \text{if }A\in \beta\uparrow \SA_{-}^{\beta},\\
	\Hom_{S^{\beta}_k}(\beta S^{\beta}_k,S^{\beta}_k)\\
	\qquad\qquad\qquad\qquad\qquad \text{if }A\in \beta \downarrow \SA_{-}^{\beta},\\
		0	\qquad\qquad\qquad\qquad\text{else}.
		\end{cases}
\end{align*}

\begin{lemma}\label{q0dual}
$Q_0$ is tilting self-dual up to a shift by $-2l(w_0)$, which means in formula
\begin{equation*}
	[\CC \circ \cD] Q_0\cong Q_0\{-2l(w_0)\}.
\end{equation*}
\end{lemma}

\begin{proof}
		We will explicitly define the isomorphism between those two objects. Let $\delta= \prod_{\alpha \in R^{+}} \alpha$ and the morphisms 
		\begin{equation*}
			[\CC \circ \cD] Q_0 \overset{g}{\underset{f}{\rightleftarrows}} Q_0\{-2l(w_0)\}
		\end{equation*}
be given by 
	\begin{align*}
		f_{A_w} \colon Q_0 (A_w) &\to [\CC \circ \cD] Q_0 \{-2l(w_0)\}(A_w),\\
		q &\mapsto (-1)^{l(w)} \delta^{-1} q \cdot\id_{S_k^{\emptyset}}
	\end{align*}
and 
		\begin{align*}
		g_{A_w} \colon [\CC \circ \cD] Q_0 (A_w) &\to  Q_0\{-2l(w_0)\} (A_w),\\
		\varphi &\mapsto (-1)^{l(w)} \delta \varphi(1).
	\end{align*}
	Indeed, both compositions yield the identity. Now the remaining question is if they are well-defined morphisms. We will study this question case by case. 
	
		\ul{Case 1} ($A \in \SA_{-}^{\beta}$): 
		This is the case where the signs in $f$ and $g$ really matter. Let $\varphi$ lie in
		\begin{equation*}
\Hom_{S^{\beta}_k}(\{(\beta x + y,y)\mid x,y\in S^{\beta}_k\}, S^{\beta}_k)\hookrightarrow [\CC \circ \cD]Q_0 (A) \oplus [\CC \circ \cD]Q_0 (\beta \uparrow A)
		\end{equation*}
		where it is embedded to
		\begin{equation*}
		(\varphi_1,\varphi_2)=(\varphi(\beta,0)\beta^{-1}\cdot \id_{S_k^{\emptyset}}, (\varphi(1,1)-\varphi(\beta,0)\beta^{-1})\cdot\id_{S_k^{\emptyset}}).
		\end{equation*}
		If $A=A_w$ and $\beta \uparrow A=A_v$ then 
		\begin{equation*}
			\{(-1)^{l(w)}, (-1)^{l(v)}\}=\{1,-1\}.
		\end{equation*} 
		By this fact, applying $g_A\oplus g_{\beta \uparrow A}$ yields
		\begin{equation*}
			(x,y)=(\pm \delta \varphi(\beta,0)\beta^{-1}, \mp \delta (\varphi(1,1)- \varphi (\beta,0) \beta^{-1}))\in S^{\beta}_k\oplus S^{\beta}_k
		\end{equation*}
		which additionally satisfy $x-y=\pm \delta \varphi(1,1) \in \beta S^{\beta}_k$. 
		
		Now, let $(\beta x + y, y)\in Q_0 (A,\beta)$. Then this is mapped by $f_A\oplus f_{\beta\uparrow A}$ to
		\begin{equation*}
			(\pm \delta^{-1}(\beta x+y)\cdot\id_{S^{\emptyset}_k}, \mp \delta^{-1}y\cdot\id_{S^{\emptyset}_k}).
\end{equation*}					
Applying this element of $[\CC \circ \cD]Q_0 (A)\oplus [\CC \circ \cD]Q_0 (\beta\uparrow A)$ to $(t,t+\beta s)$ yields
\begin{equation*}
	\pm \delta^{-1} \beta (xt-ys) \in S^{\beta}_k.
\end{equation*}

	\ul{Case 2} ($A \in \beta \uparrow\SA_{-}^{\beta}$ and $\Hom_{S^{\beta}_k}(S^{\beta}_k,S^{\beta}_k)\hookrightarrow [\CC \circ \cD] Q_0 (A)$):
	
		Let $\varphi \in \Hom_{S^{\beta}_k}(S^{\beta}_k, S^{\beta}_k)$ and map this by $g_A$ to
		\begin{equation*}
			(-1)^{l(w)}\delta \varphi(1)\in  \beta S^{\beta}_k
		\end{equation*}
		because $\varphi(1) \in S^{\beta}_k$. Consider now the map 
		\begin{equation*}
			f_A(\beta s)=(-1)^{l(w)}\delta^{-1} \beta s \cdot \id_{S^{\emptyset}_k}.
		\end{equation*}
		Restricting this to $S^{\beta}_k$ yields an element in $\Hom_{S^{\beta}_k}(S^{\beta}_k,S^{\beta}_k)$.
		
		\ul{Case 3} ($A \in \beta \downarrow \SA_{-}^{\beta}$ and $\Hom_{S^{\beta}_k}(\beta S^{\beta}_k,S^{\beta}_k)\hookrightarrow  [\CC \circ \cD]Q_0 (\beta \uparrow A)$):
		
		Let $\varphi \in \Hom_{S^{\beta}_k}(\beta S^{\beta}_k, S^{\beta}_k)$ and map this by $g_{\beta \uparrow A}$ to
		\begin{equation*}
			(-1)^{l(w)}\delta \varphi(\beta)\beta^{-1}\in S^{\beta}_k
		\end{equation*}
		because $\varphi(\beta) \in S^{\beta}_k$ and $\delta \beta^{-1} \in S^{\beta}_k$. Consider now the map 
		\begin{equation*}
			f_{\beta \uparrow A}(s)=(-1)^{l(w)}\delta^{-1} s \cdot \id_{S^{\emptyset}_k}.
		\end{equation*}
		Restricting this to $\beta S^{\beta}_k$ yields an element in $\Hom_{S^{\beta}_k}(\beta S^{\beta}_k,S^{\beta}_k)$.		
	\end{proof}

\subsection{$\CM_k^{\circ}$ is a subcategory of $\CM_k$}
Let $s \in \CS \subset \CW$ be a simple reflection. Then $Q_0$ is special in the sense that for each alcove $A\in \SA$ the $S^{\emptyset}_k$-modules $Q_0(\wm{A}{s})$ and $Q_0(\wp{A}{s})$ coincide (this is definitely false for the remaining affine reflection $\widehat{s}\in \widehat{\CS}\setminus \CS$). Hence, there is the diagonal embedding
\begin{align*}
\Delta_A \colon Q_0(A)&\to Q_0(\wm{A}{s})\oplus Q_0(\wp{A}{s})=\CT_s Q_0(A)\\
t&\mapsto (t,t)
\end{align*}
and this even defines a morphism $\Delta$ from $Q_0$ to $\CT_s Q_0$ as we will see in
\begin{lemma}\label{homsvonq}
Let $K\in \CK_k$, $f\in \Hom_{\CK_k}(Q_0, K)$ and $s\in \CS$. Then $\w{f}{s} = (\CT_s f) \circ \Delta \in \Hom_{\CK_k}(Q_0, \CT_s K)$.
\end{lemma}

\begin{proof}
We need to verify that $\Delta \in \Hom_{\CK_k}(Q_0, \CT_s Q_0)$. If the wall of $A\in \SA$ is not of type $\beta$ then there is no hyperplane $H_{\beta,n}$ separating $\wm{A}{s}$ and $\wp{A}{s}$. This means 
\begin{equation*}
[\beta \uparrow \w{A}{s}\neq \w{A}{s}\text{ and } A\in \SA_-^{\beta}]\Rightarrow \{\wm{A}{s},\wp{A}{s}\}\subset \SA_-^{\beta}
\end{equation*}
where $\SA_-^{\beta}=\{A \in \SA\mid A=A_w \text{ with } w\in \CW, A \subset H_{\beta}^{-}\}$ as in \eqref{abminus} 
and equivalent statements for $A\in \SA_+^{\beta}$ and $A\in \beta \downarrow\SA_-^{\beta}$. As $\wm{A}{s}$ and $\wp{A}{s}$ share the same \emph{$\beta$-strip}, i.e. there is a $n\in \ZZ$ such that $\wm{A}{s},\wp{A}{s}\subset H^{-}_{\beta,n+1}\cap H^{+}_{\beta,n}$, the $S_{k}^{\beta}$-modules $Q_0(\wm{A}{s},\beta)$ and $Q_0(\wp{A}{s},\beta)$ can be identified canonically. Hence, for $(p,q)\in Q_0(A,\beta)$ the element $(\Delta_A \oplus \Delta_{\beta \uparrow A})(p,q)\in \CT_s Q_0(A,\beta)$. 

Furthermore, observe that for $\beta \uparrow \w{A}{s}=\w{A}{s}$ by definition of $\SA_-^{\beta}$ there are the implications
\begin{align*}
A\in \SA_-^{\beta} &\Rightarrow A=\wm{A}{s} \text{ and }\\
[A\in \SA_+^{\beta} \text{ or } A\in \beta \downarrow \SA_-^{\beta}] &\Rightarrow A=\wp{A}{s}.
\end{align*}
With this in mind we will verify the remaining cases. If $\beta\uparrow \w{A}{s}=\w{A}{s}$, $A=\wm{A}{s}$ and $(m,n) \in Q_0(A,\beta)$ then by definition of $Q_0$ we know that $(m,m), (n,n)\in Q_0(A,\beta)$ and $(m-n,m-n) \in \beta\, Q_0(A,\beta)$. Hence, $(m,m,n,n)\in Q_0(A,\beta)$. 

As $(1,1)S_k^{\beta}\subset \{(\beta x+y,y)\mid x,y\in S_k^{\beta}\}$ we obtain 
\begin{equation*}
(\Delta_A \oplus \Delta_{\beta\uparrow A}) (Q_0(A,\beta))\subseteq [\CT_s Q_0](A,\beta)
\end{equation*}
even in the case when $\beta \uparrow \w{A}{s}=\w{A}{s}, A=\wp{A}{s}$. Therefore, $\Delta$ is indeed a morphism. 

But if $\Delta\in \Hom_{\CK_k}(Q_0,\CT_s Q_0)$, then 
$(\CT_s f) \circ \Delta \in \Hom_{\CK_k}(Q_0, \CT_s K)$ since $(\CT_s f)\in \Hom_{\CK_k}(\CT_s Q_0, \CT_s K)$.
\end{proof}

This lemma is very helpful for defining various morphisms with domain $Q_0$. One application can be seen in 

\begin{proposition}\label{q0summand}
$Q_0$ is an object in $\CM_k$. In particular, $\CM_k^{\circ}$ is a full subcategory of $\CM_k$. 
\end{proposition}

\begin{proof}
Let $(s_1,\ldots, s_l)$ be a reduced expression for $w_0$, the longest element of the finite Weyl group. We want to show that $Q_0$ is a direct summand of $\CT_{s_l}\circ \ldots \circ\CT_{s_1}(P_0)$.
In particular, for all $i \in \{1,\ldots, l\}$ the reflection $s_i$ is in $\CS$ so that we can make use of Lemma \ref{homsvonq}. We start with two morphisms with domain $Q_0$, namely $f \in \Hom_{\CK_k}(Q_0, P_0)$ given by $f_{A_e}=\id_{S_k^{\emptyset}}$ and $g \in \Hom_{\CK_k}(Q_0, \CC P_0)$ given by $g_{w_0.A_{e}}=\id_{S_k^{\emptyset}}$. Thereon, we continue applying Lemma \ref{homsvonq} until we get the morphisms
\begin{align}
\tilde{f}\colon Q_0 &\to \CT_{s_l}\circ \ldots \circ\CT_{s_1}(P_0),\label{homnachp}\\
\tilde{g}\colon Q_0 &\to [\CT_{s_l}\circ \ldots \circ\CT_{s_1}](\CC P_0).\label{homnachgekippt}
\end{align}
In a next step, we dualize the morphism in \eqref{homnachgekippt} to 
\begin{equation*}
[\CC\circ\cD] (\tilde{g})\colon \CC\circ\cD\circ\CT_{s_l}\circ \ldots \circ\CT_{s_1}\circ\CC (P_0)\to \CC\circ\cD (Q_0).
\end{equation*}
With $\cD P_0 \simeq P_0$ and following Proposition \ref{theo:dualtrans} and Proposition \ref{theo:kipptrans}, there is an isomorphism 
\begin{equation*}
\eta \colon \CT_{s_l}\circ \ldots \circ\CT_{s_1} (P_0) \to \CC\circ\cD\circ\CT_{s_l}\circ \ldots \circ\CT_{s_1}\circ\CC (P_0)\{2l\}
\end{equation*}
and by Lemma \ref{q0dual} there is an isomorphism
\begin{equation*}
\xi \colon \CC\circ \cD (Q_0)\{2l\}\to Q_0.
\end{equation*}
Then 
\begin{equation*}
\xi_{A_{w_0}}\circ [\CC\circ\cD] (\tilde{g})_{A_{w_0}}   \circ \eta_{A_{w_0}}=\xi_{A_{w_0}}\circ \id_{S_k^{\emptyset}}\circ \eta_{A_{w_0}}
\end{equation*}
is an isomorphism on the $A_{w_0}$-component, too. (Doing explicit calculation even shows that this is the identity.) 

An endomorphism $h$ of $Q_0$ which is an isomorphism of degree $0$ at the $A_{w_0}$-component cannot be nilpotent. As $Q_0$ is indecomposable, $h$ must be an automorphism.

We obtain the sequence
\begin{equation*}
Q_0 \to \CT_{s_l}\circ \ldots \circ\CT_{s_1}(P_0) \to Q_0
\end{equation*} 
which builds an automorphism of $Q_0$. So $Q_0$ is indeed a direct summand of $\CT_{s_l}\circ \ldots \circ\CT_{s_1}(P_0) $.
\end{proof}

We can define $Q_0$ equivalently as the indecomposable direct summand of $\CT_{s_l}\circ \ldots \circ\CT_{s_1}(P_0)$ with $w_0 \in \supp \, Q_0$ where $(s_1,\ldots,s_l)$ is a reduced expression of $w_0$. This summand must be unique because the rank of $[\CT_{s_l}\circ \ldots \circ\CT_{s_1}]P_0(A_{w_0})$ is one. Then another consequence of the above proposition is that this definition does not depend on the choice of the reduced expression.

\subsection{Bott-Samelson like objects}
\begin{theorem}\label{dual:mainthm}
For a Bott-Samelson like object $M\in \CM_k^{\circ}$ there is an $z \in \ZZ$ such that
\begin{equation*}
		\CC \circ \cD (M) \cong M \{z\}.
\end{equation*}
\end{theorem}

\begin{proof}
This proof uses the recursive definition of Bott-Samelson like objects (see \eqref{bottsamlike}) in $\CM_k$. It consists of applying Lemma \ref{q0dual} and inductively making use of Proposition \ref{theo:dualtrans} and Proposition \ref{theo:kipptrans}. The integer $z$ in particular depends on the length of the word in \eqref{bottsamlike} -- it equals 
\begin{equation*}
z= -2r-2l(w_0) -2n.
\end{equation*}
\end{proof}

\subsection{Indecomposables corresponding to the anti-fundamental box}
For the next steps we need the notion of \emph{$\alpha$-strings} in the alcove pattern. Let $\alpha \in R^+$ and $A\in \SA$. Then
\begin{equation*}
\{\alpha \uparrow^n A \mid n \in \ZZ\}\subset \SA
\end{equation*}
is such an $\alpha$-string. Consider a Bott-Samelson like object $M$ in $\CM_k^{\circ}$, i.e.
\begin{equation*}
M=\CT_{s_l}\circ\ldots\circ \CT_{s_1} (Q_0).
\end{equation*}
The intersection of an $\alpha$-string with the support of this object $M$ is \emph{connected}:
\begin{lemma} Let $M$ be a Bott-Samelson like object. For all $\alpha \in R^+$ and $A\in \supp\, M$ there are $k\in \NN$ and $B\in \SA$ such that
\begin{equation*}
\supp\, M \cap \{\alpha \uparrow^n A \mid n \in \ZZ\} =\{B,\alpha \uparrow B,\ldots, \alpha \uparrow^k B\}.
\end{equation*}
\end{lemma}
\begin{proof}
This is true for $Q_0$ since for all $\alpha \in R^+$ we have $\supp\, Q_0 =\SA_-^{\alpha} \,\dot\cup\,\SA_+^{\alpha}$. If $A\in \SA_-^{\alpha}$ then 
\begin{equation*}
\supp\, Q_0 \cap \{\alpha \uparrow^n A \mid n \in \ZZ\}=\{A, \alpha \uparrow A\}.
\end{equation*}
If $A\in \SA_+^{\alpha}$ then 
\begin{equation*}
\supp\, Q_0 \cap \{\alpha \uparrow^n A \mid n \in \ZZ\}=\{\alpha \downarrow A, A\}.
\end{equation*}
Assume the claim holds for a Bott-Samelson like object $M'$. Consider now $M=\CT_s M'$ and let $C, \alpha \uparrow^k C\in \supp\,M$. Then the following two intersections are non-empty
\begin{align*}
\{\wm{C}{s},\wp{C}{s}\}\cap \supp \, M'\neq \emptyset,\\
\{\wm{(\alpha\uparrow^k C)}{s},\wp{(\alpha\uparrow^k C)}{s}\}\cap \supp \, M'\neq \emptyset.
\end{align*}
Without loss of generality let $\wp{C}{s} \in \supp \, M'$ then so is $s_{\alpha}\wp{C}{s}\in \supp \, M'$ where both live in the same $\alpha$-string. By induction hypothesis the intersection of the $\alpha$-string through $\wp{C}{s}$ with $\supp \, M' \subset \supp \, M$ is connected. Translating this connected $\alpha$-string by $s$ gives another one between $\wm{C}{s}$ and $\w{(s_{\alpha}\w{C}{s})}{s}_{\mp}$ in the support of $M$. 
If $\alpha\uparrow^k C$ was not in the translated $\alpha$-string yet then start the same procedure with $\w{\alpha\uparrow^k C}{s}_{\pm}$. We obtain
\begin{equation*}
\{C,\ldots, \alpha \uparrow^k C\}\subseteq \supp\, M \cap \{\alpha \uparrow^n A \mid n \in \ZZ\}
\end{equation*}
and with the finiteness of the support of $M$ we receive the claim.
\end{proof}

We will now consider the special case when the support of a Bott-Samelson like object intersects an $\alpha$-string in exactly two alcoves.
\begin{lemma}\label{link}
Let $M \in \CM^{\circ}_k$ be a Bott-Samelson like object and $B \in \SA$ with 
\begin{equation*}
|\supp\, M \cap \{\alpha \uparrow^n B \mid n \in \ZZ\}| =2. 
\end{equation*}
Let $A\in \SA$ be such that
\begin{equation*}
\supp\, M \cap \{\alpha \uparrow^n B \mid n \in \ZZ\}=\{A,\alpha \uparrow A\}. 
\end{equation*}
If $\rk\, M(A)=1$, then $M$ links the alcoves $A$ and $\alpha\uparrow A$. In particular, $\{A,\alpha \uparrow A\}\subseteq \supp \, N$ where $N$ is the indecomposable summand of $M$ with $A\in \supp \, N$. 
\end{lemma}
\begin{proof}
Since $\alpha \downarrow A \notin \supp \, M$, the $S^{\alpha}_k$-module $M(\alpha \downarrow A,\alpha)$ must be of the form $\alpha^l  S^{\alpha}_k \hookrightarrow 0\oplus S_k^{\emptyset}$. We obtain
\begin{align*}
M(\alpha\downarrow A,\alpha) \cap M(A) =\alpha^l S^{\alpha}_k,\\
\pr_A [M(\alpha \downarrow A, \alpha)] =\alpha^l S^{\alpha}_k.
\end{align*}
But $M \in \CM_k^{\circ}$ and we apply Lemma \ref{verma:imagel}:
\begin{align*}
\pr_A [M(A,\alpha)]&=\alpha^l  S^{\alpha}_k,\\
M(A,\alpha) \cap M(A)&=\alpha^{l+1}  S^{\alpha}_k.
\end{align*}
This is only true if $M$ links $A$ to $\alpha \uparrow A$. 
\end{proof}

Our goal is to extend the self-duality we have seen in the context of Bott-Samelson like objects in $\CM_k^{\circ}$ to certain indecomposable objects. The \emph{anti-fundamental box} $\Pi^-$ of $\SA$ consists of those alcoves that live in the strips
\begin{equation*}
A\subset H^-_{\alpha,0}\cap H^+_{\alpha,-1}
\end{equation*}
for all $\alpha \in \Delta$. For each $A\in\SA$ there is exactly one $\lambda \in X^{\vee}$ such that $\lambda + A \in \Pi^-$.

\begin{theorem}\label{selfdualanti}
Let $A_w \in \Pi^-$ and $(s_1,\ldots, s_l, s_{l+1},\ldots, s_j)$ be a reduced expression of $w\in \widehat{\CW}$ with $w_0=s_1 \ldots s_l$. Define 
\begin{equation*}
M=[\CT_{s_j}\circ \ldots \circ \CT_{s_{1}}]P_0.
\end{equation*}
Let $N$ be the up to isomorphism unique indecomposable direct summand of $M$ with $A_w\in \supp\, N$. Then 
\begin{equation*}
\CC \circ \cD (N)\cong N\{-2l(w)\}.
\end{equation*}
\end{theorem}

\begin{proof}
With the proof of Proposition \ref{q0summand} the object $Q_0$ is a direct summand of $[\CT_{s_l}\circ \ldots \circ \CT_{s_{1}}]P_0$ and so $[\CT_{s_j}\circ \ldots \circ \CT_{s_{l+1}}]Q_0=M'$ is a summand of $M$. Hence, the unique indecomposable direct summand of $M'$ with $A_w$ in its support is isomorphic to $N$, too. The rank in an alcove component of a Bott-Samelson like object is invariant under the action of the finite Weyl group which can be seen by induction. First of all, it is
\begin{equation*}
\rk\, Q_0(A_x) =\begin{cases}
1 & \text{if }x\in \CW,\\
0 & \text{else}
\end{cases}
\end{equation*}
and the induction step is achieved with the fact that $\{w.\wm{A}{s},w.\wp{A}{s}\}=\{\wm{(w.A)}{s}, \wp{(w.A)}{s}\}$. Hence for all $w\in \CW$ and $A \in \SA$ we have
\begin{equation}\label{rankinvariant}
\rk\, M'(A) =\rk\, M'(w.A).
\end{equation}
As $(s_1,\ldots,s_j)$ is a reduced expression of $w$, $\rk\, M(A_w)= \rk\, M'(A_w)=1$ and hence, $\rk\, M'(w_0. A_w)=1$. For this reason, there is a unique indecomposable summand $N_{\CC}$ of $M'$ with $w_0. A_w \in \supp\, N_{\CC}$. Because of Theorem \ref{dual:mainthm} and $\supp \, \CC N=w_0.\, \supp\, N$ there is 
\begin{equation*}
\CC \circ \cD (N)\cong N_{\CC}\{-2l(w)\}.
\end{equation*}
In the remaining part of the proof we will show that $w_0 A_w \in \supp \, N$. 

\begin{figure}[htbp] 
  \centering
  \includegraphics[width=0.7\textwidth]{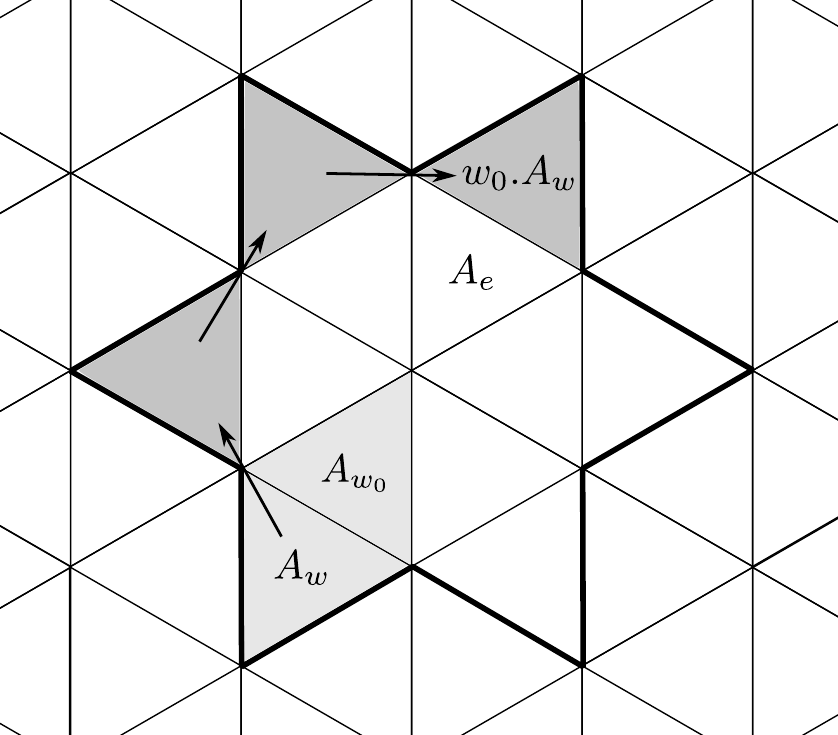}
  \caption{The path $A_w \to w_0.A_w$, the support of $M$ is framed}
  Color key \quad
\includegraphics{artcp2.pdf}\- Alcoves in $\Pi^-$ \quad
\includegraphics{artcp1.pdf}\- Constructed path $A_w \to w_0.A_w$
  \label{fig:path}
\end{figure}

We would like to construct a path via the $(\cdot \uparrow)$-operation through the support of $M'$ between the alcoves $A_w$ and $w_0 A_w$. For an example of type $\tilde{A}_2$, see Figure \ref{fig:path}. Recall that $(s_1,\ldots, s_l)$ was a reduced expression of $w_0$. Consider
\begin{equation}\label{path}
A_{w} \to s_1 .A_{w} \to s_1 s_{2} .A_{w} \to \ldots \to s_1 \ldots s_l .A_{w}=w_0 .A_w.
\end{equation}
As the alcove $A_w$ is in the anti-fundamental box, it lives in
\begin{equation*}
A_w \subset H_{\alpha,0}^-\cap H_{\alpha,-1}^+
\end{equation*}
for all simple roots $\alpha\in \Delta$. Hence for each $\alpha\in \Delta$ and $k\in {1,\ldots,l}$
\begin{equation*}
(s_1 \ldots s_{k-1})A_w \subset H_{s_1 \ldots s_{k-1}(\alpha),0}^-\cap H_{s_1 \ldots s_{k-1}(\alpha),-1}^+,
\end{equation*}
in particular when $\alpha$ corresponds to the simple root of the simple reflection $s_k$ noted by $\alpha_k$. Furthermore, it is known that there is a positive root $\beta_k \in R^+$ such that
\begin{equation*}
s_1 \ldots s_{k-1} s_k =s_{\beta_k} s_1\ldots s_{k-1}.
\end{equation*}
From this follows 
\begin{equation}\label{more:dual}
s_1 \ldots s_{k-1} (\alpha_k) \in \{\pm \beta_k\}.
\end{equation}
Since this describes the walk from $A_{w_0}$ to $A_e$ at the same time, all roots in \eqref{more:dual} are positive. 
We have got $s_1 \ldots s_{k-1} (\alpha_k) = \beta_k$ and on that account 
\begin{equation*}
s_1 \ldots s_{k-1}.A_w \subset H_{\beta_k,0}^-\cap H_{\beta_k,-1}^+.
\end{equation*}
Therefore
\begin{equation*}
s_1\ldots s_k.A_w=s_{\beta_k}s_1\ldots s_{k-1}.A_w=\beta_k \uparrow (s_1\ldots s_{k-1}.A_w).
\end{equation*}
So \eqref{path} indeed yields the desired path, represented by
\begin{equation*}
A_{w} \to \beta_1\uparrow A_{w} \to  \ldots \to  \beta_l\uparrow \ldots \beta_1 \uparrow A_{w}=w_0 A_w.
\end{equation*}

For any simple root $\alpha \in \Delta$ we would like to describe the $\alpha$-string through $A_w$ intersected with the support of $M'$.  Certainly, $\alpha \uparrow A_w=s_{\alpha} A_w \in \supp \, M'$. On the other side $\alpha \downarrow A_w \notin \supp \, M'$ as in the anti-fundamental chamber 
\begin{equation*}
l(\alpha \downarrow A_w) \geq l(A_w).
\end{equation*}
The $\alpha$-string of $A_w$ in $M'$ additionally does not contain $s_{\alpha}(\alpha \downarrow A_w)=\alpha \uparrow^2 A_w$. This leads to 
\begin{equation*}
 \supp \,M' \cap \{\alpha \uparrow^z A_w\mid z \in \ZZ\}=\{A_w, \alpha \uparrow A_w\}.
\end{equation*} 
This statement can be extended to all alcoves in the path we would like to study. Applying \eqref{rankinvariant} in shape of $w=s_1\ldots s_n$ with $n\in\{1,\ldots, l\}$ shows that the $\beta_k$-strings in $s_1\ldots s_{k-1}A_w$ all own exactly two elements. But now we can apply Lemma \ref{link} inductively and receive that $w_0 .A_w \in \supp \,N$. 
\end{proof}

\begin{remark}
So far, we do not have an intrinsic classification of the indecomposable objects in $\CM_k$ nor in $\CM_k^{\circ}$. However,  the representation theoretical side says that it is enough to know the self-duality of objects as in the previous Theorem  \ref{selfdualanti}. Other indecomposable objects in $\CM_k^{\circ}$ are received by translating those with respect to the alcove pattern and hence are tilting self-dual up to translating inside the alcove pattern. 
\end{remark}

\subsection*{Acknowledgements}
I would like to thank Peter Fiebig for very helpful discussions and ideas and for his useful comments on previous versions of this paper. This project is supported by the DFG priority program 1388.

\bibliography{steglich}

\begin{thebibliography}{Hum90}

\bibitem[AJS94]{MR1272539}
H.~H. Andersen, J.~C. Jantzen, and W.~Soergel.
\newblock Representations of quantum groups at a {$p$}th root of unity and of
  semisimple groups in characteristic {$p$}: independence of {$p$}.
\newblock {\em Ast\'erisque}, (220):321, 1994.

\bibitem[Fie11]{MR2726602}
Peter Fiebig.
\newblock Sheaves on affine {S}chubert varieties, modular representations, and
  {L}usztig's conjecture.
\newblock {\em J. Amer. Math. Soc.}, 24(1):133--181, 2011.

\bibitem[Hum90]{humphreys1990reflection}
James~E Humphreys.
\newblock {\em Reflection Groups and Coxeter Groups. Number 29 in Cambridge
  studies in advanced mathematics}.
\newblock Cambridge University Press, 1990.

\bibitem[KL93]{KL93}
David Kazhdan and George Lusztig.
\newblock Tensor structures arising from affine {L}ie algebras.
\newblock {\em J. Amer. Math. Soc.}, I-IV, 1993.

\bibitem[KT95]{zbMATH00755403}
Masaki {Kashiwara} and Toshiyuki {Tanisaki}.
\newblock {Kazhdan-Lusztig conjecture for affine Lie algebras with negative
  level.}
\newblock {\em {Duke Math. J.}}, 77(1):21--62, 1995.

\bibitem[Lus80]{lusztig1980some}
George Lusztig.
\newblock Some problems in the representation theory of finite {C}hevalley
  groups.
\newblock In {\em Proc. Symp. Pure Math}, volume~37, pages 313--317, 1980.

\end{thebibliography}
\bibliographystyle{alpha}

\end{document}